\documentclass[10pt,a4paper,reqno]{amsart}
\usepackage[foot]{amsaddr}

\usepackage{amsmath,amsfonts,amsthm,epsfig,graphicx,mathtools,tikz-cd, amssymb,setspace,enumerate,enumitem}
\usepackage[foot]{amsaddr}
\usepackage{caption}
\usepackage{mathrsfs,color}
\usepackage[margin=1in]{geometry}
\usepackage{mathabx}

\usepackage{mathtools}
\mathtoolsset{showonlyrefs}
\usepackage[page,header]{appendix}
\usepackage{titletoc}

\newcommand{\rd}{\,\mathrm{d}}
\numberwithin{equation}{section}
\newtheorem{theorem}{Theorem}[section]
\newtheorem{lemma}[theorem]{Lemma}

\newtheorem{proposition}[theorem]{Proposition}

\newtheorem{remark}[theorem]{Remark}

\usepackage{apptools}
\AtAppendix{\counterwithin{theorem}{section}}

\usepackage{mathtools}

\def\bx{{\bf x}}
\def\bbx{{\bf \bar{x}}}
\def\bby{{\bf \bar{y}}}
\def\bxi{\bar{\xi}}
\def\by{{\bf y}}

\def\cE{\mathcal{E}}
\def\cB{\mathcal{B}}

\def\cW{\mathcal{W}}

\def\cF{\mathcal{F}}

\def\kS{\mathfrak{S}}

\def\supp{\textnormal{supp\,}}

\def\essinf{\textnormal{essinf\,}}

\begin{document}

\title{Minimizers of 3D anisotropic interaction energies}

\author{Jos\'e A. Carrillo$^{\dagger}$, Ruiwen Shu$^{\dagger}$}
\date{\today}
\subjclass[2020]{}
\address[$\dagger$]{Mathematical Institute, University of Oxford, Oxford OX2 6GG, UK. Emails: {\tt carrillo@maths.ox.ac.uk, shu@maths.ox.ac.uk}}

\maketitle

\begin{abstract}
We study a large family of axisymmetric Riesz-type singular interaction potentials with anisotropy in three dimensions. We generalize some of the results of the recent work \cite{CS22} in two dimensions to the present setting. For potentials with linear interpolation convexity, their associated global energy minimizers are given by explicit formulas whose supports are ellipsoids. We show that for less singular anisotropic Riesz potentials, the global minimizer may collapse into one or two dimensional concentrated measures which minimize restricted isotropic Riesz interaction energies. Some partial aspects of these questions are also tackled in the intermediate range of  singularities in which one dimensional vertical collapse is not allowed.
Collapse to lower dimensional structures is proved at the critical value of the convexity but not necessarily to vertically or horizontally concentrated measures, leading to interesting open problems. 
\end{abstract}

\section{Introduction}
In this work, we focus on the analysis of the 3D anisotropic interaction energy
\begin{equation}
    E[\rho] = \frac{1}{2}\int_{\mathbb{R}^3}\int_{\mathbb{R}^3}W(\bx-\by)\rho(\by)\rd{\by}\rho(\bx)\rd{\bx}
\end{equation}
where $\rho$ is a probability measure on $\mathbb{R}^3$ and $W$ is the anisotropic interaction potential
\begin{equation}\label{W}
    W(\bx) = |\bx|^{-s}\Omega(\bbx)+|\bx|^2,\quad 0<s<3\,.
\end{equation}
Here we denote $\bbx=\frac{\bx}{|\bx|}\in S^2$ as the angle variable\footnote{This `bar' notation should be distinguished from the commonly used notation for complex conjugates, the latter never appears in this paper.}, and the anisotropic part $\Omega$ is defined on $S^2$. We will always assume $\Omega$ satisfies
\begin{equation}
    \text{{\bf (H)}: $\Omega$ is smooth, strictly positive, and $\Omega(\bbx)=\Omega(-\bbx)$.}
\end{equation}
For such $\Omega$, $W$ satisfies the condition
\begin{equation}\begin{split}
    & \text{{\bf (W)}: $W$ is even, locally integrable, lower-semicontinuous, bounded from below,} \\ & 
    \text{and bounded above and below by positive multiples of $|\bx|^{-s}$ near 0.}
\end{split}\end{equation}
As a result, $E[\rho]$ is well-defined with values in $\mathbb{R}\cup\{+\infty\}$ for any probability measure $\rho$.
We will also use the following less strict assumptions
\begin{equation}
    \text{{\bf (H0)}: $\Omega$ is smooth, and $\Omega(\bbx)=\Omega(-\bbx)$.}
\end{equation}
in some situations, and in particular, we will study the parametrized potential
\begin{equation}\label{Walpha}
    W_\alpha(\bx) = |\bx|^{-s}(1+\alpha\omega(\bbx))+|\bx|^2,\quad \alpha\ge 0
\end{equation}
with  $\omega$ satisfying
\begin{equation}
    \text{{\bf (h)}: $\omega$ is smooth, $\min \omega = 0$ and $\omega(\bbx)=\omega(-\bbx)$.}
\end{equation}
In a recent result \cite{CS22}, the authors have analysed the 2D anisotropic interaction energies in detail. The main objective of the present work is to generalize the strategy and techniques developed in \cite{CS22} to the three dimensional case. Our approaches work for general dimensions $d\geq 4$ with $d-4<s<d$, which we will not treat to avoid cumbersome technicalities.

One of the basic tools used in \cite{CS22} is the Linear Interpolation Convexity (LIC) of the interaction energy functional $E[\rho]$ defined as: for any two compactly supported probability measures $\rho_0\ne \rho_1$ with the same center of mass, the energy along their linear interpolation curve $E[(1-t)\rho_0+t\rho_1],\,t \in [0,1]$ is always strictly convex.
We proved in \cite{CS22} that strict LIC is essentially equivalent to nonnegativity of the Fourier transform of $W$ for potentials of the form \eqref{W}.
LIC was utilised before for uniqueness of global miminizers and  \cite{Lieb81,lopes2017uniqueness,MRS19,davies2021classifying,davies2021classifying2,frank2021minimizers} and uniqueness of Wasserstein-$\infty$ local minimizers in \cite{carrilloshu21}.

To see whether LIC holds for an anisotropic interaction potential, we follow a similar strategy in 3D as in the 2D case by studying the Fourier transform of the potential $W$ in details. Our main strategy for identifying global minimizers for anisotropic potentials in the LIC case is then quite direct: it suffices to check that the candidate satisfies the Euler-Lagrange conditions obtained in \cite{BCLR2} whenever the potential energy is LIC.

Let us recall some of the previous results for anisotropic potentials. The collapse to one dimensional minimizers was shown to happen in 2D for the particular case 
\begin{equation}\label{particular}
    W_{\log,\alpha}(\bx) = -\ln|\bx| + \alpha \frac{x_1^2}{|x|^2} + |\bx|^2,\quad\alpha\ge 0 ,
\end{equation} 
in the seminal paper \cite{MRS19} at the value of $\alpha$ for which the potential ceases to be LIC.  A series of recent works \cite{CMMRSV,CMMRSV2,MMRSV20,MMSRV21-1,MMSRV21-2} study this particular family of anisotropic potentials and small perturbations, showing that the the collapse to one dimensional vertical minimizers happens for $\alpha$ large enough and showing that the minimizers for smaller values than the critical LIC convexity value are given by the characteristic function of suitable ellipses. 

We recently showed in \cite{CS22} that a similar behavior in 2D is shared by a large family of potentials of the form \eqref{Walpha} with $0<s< 2$ as well as its logarithmic counterpart, generalizing the results in the previous papers. More precisely, we proved that the minimizers in the LIC range are given by push-forward of the global minimizer for isotropic Riesz potentials to suitable ellipses, showing that the ellipse-shaped minimizers are indeed generic in the LIC range. The collapse to the one dimensional vertical minimizer was also obtained for this family, showing that this again happens for $\alpha$ large enough. Moreover, we showed that generically there is a gap in between these two behaviors in which one dimensional structures appear but not necessarily are supported on vertical lines. This is based on the concept of infinitesimal concavity introduced in \cite{carrilloshu21}: finding a signed measure $\mu$ with zero mass and center of mass and arbitrarily small support such that $E[\mu]<0$. Infinitesimal concavity of an interaction energy $E[\rho]$ is a signature of collapsing to lower dimensions since we proved in \cite{carrilloshu21} that their global minimizers contain no interior points for any superlevel set of their absolute continuous parts.

Concerning the 3D case, the only available result  \cite{CMMRSV2} dealt with the one-parameter family of potentials $-\frac{1}{|\bx|} + \alpha\frac{x_1^2}{|\bx|^3} + |\bx|^2$. It showed that the global minimizer in the LIC range $-1< \alpha \le 1 $ is the characteristic function of an ellipsoid, but the behavior beyond this range was not discussed. The main objective of the present work is to generalize the results in \cite{CS22} from 2D to 3D, thus treating 3D axisymmetric potentials of the form \eqref{W} and \eqref{Walpha} with general singularities $s$ and angular profiles $\Omega$. We are able to generalize two types of results. On one hand, the LIC property of the potentials leads to ellipsoid-shaped global minimizers characterized by the push-forward of the minimizers for 3D isotropic Riesz potentials. On the other hand,  infinitesimal concavity results in the collapse to lower dimensional structures. For suitable singularity of the potential and the angular profile, such structures for large values of $\alpha$ are known to be the 1D/2D minimizers for restricted isotropic Riesz potentials. More complicated behavior is also possible, including expansion of the support as $\alpha\rightarrow\infty$. We now describe the main results of this paper in details.

Most of our results will be obtained using spherical coordinates that we denote as
\begin{equation}
    \bbx = \begin{pmatrix}
    \sin\theta\cos\mu \\
    \sin\theta\sin\mu \\
    \cos\theta \\
    \end{pmatrix},\quad 0\le \theta\le \pi,\quad 0\le \mu < 2\pi
\end{equation}
in the physical space, and
\begin{equation}\label{xi}
    \xi = |\xi|\bxi,\quad \bxi = \begin{pmatrix}
    \sin\varphi\cos\nu \\
    \sin\varphi\sin\nu \\
    \cos\varphi \\
    \end{pmatrix},\quad 0\le \varphi\le \pi,\quad 0\le \nu < 2\pi
\end{equation}
in the Fourier space. In Section \ref{sec_LIC} we first derive the formula for the Fourier transform for functions of the form $|\bx|^{-s}\Omega(\bbx)$. We will show in Lemma \ref{lem_FT} that for $\Omega$ satisfying {\bf (H0)} and $0<s<3$,
\begin{equation}
    \cF[|\bx|^{-s}\Omega(\bbx)] = |\xi|^{-3+s}\tilde{\Omega}(\bxi;s)
\end{equation}
for some function $\tilde{\Omega}(\bxi;s)$ smooth in the variable $\bxi\in S^2$. For $2<s<3$, its explicit formula is given by \eqref{lem_FT_2} as a convolution-type operator, and we also derive the explicit formulas for other values of $s$ in \eqref{lem_FT_2b} and \eqref{lem_FT_3}. Here, one important observation is the \emph{holomorphic} property of $\tilde{\Omega}(\bxi;s)$ with respect to the variable $s$ (when extending to all complex numbers $s$ with $0<\Re(s)<3$). It allows us to apply analytic continuation arguments to treat the values of $s$ for which we lack explicit formulas. We also review the results we obtained in \cite{CS22} related to the LIC property and state their generalizations to 3D. 

Then, in Section \ref{sec_ell} we study the energy minimizers for the interaction potential $W$ in \eqref{W}, in the case of LIC. For simplicity, as mentioned earlier, we focus on the case when $W$ is axisymmetric. For $0<s<3$, we prove that the unique energy minimizer is necessarily some $\rho_{a,b}$ as in \eqref{rhoab}, an axisymmetric rescaling of the minimizer of the corresponding 3D isotropic interaction energy (Theorem \ref{thm_ell}). In particular, the shape of its support is an ellipsoid. This result is analogous to its counterpart in \cite{CS22}. To prove this result in the case of $0<s<1$, we take a similar approach as in \cite{CS22}. The key step is to show that the potential generated by any $\rho_{a,b}$ is quadratic (Lemma \ref{lem_ab}), which is proved by the decomposition of the potential $|\bx|^{-s}\Omega(\bbx)$ into a convex combination of 1D potentials in different directions as in \eqref{decomp}. To extend it to a wider range of $s$, we use an analytic continuation argument based on the holomorphic properties established in the previous section.

The previously studied LIC cases include any potential of the form $W_\alpha$ in \eqref{Walpha} if 
\begin{itemize}
    \item Either $\tilde{\omega}$ (the angle function of the Fourier transform of $|\bx|^{-s}\omega(\bbx)$ as in \eqref{lem_FT_1}) is nonnegative. This include all the cases of $2\le s < 3$, by \eqref{lem_FT_2}.
    \item Or $\tilde{\omega}$ is not nonnegative but $\alpha$ is small. To be precise, $W_\alpha$ is LIC if and only if $0\le \alpha \le \alpha_L$, where
    \begin{equation}\label{alphaL}
        \alpha_L = \frac{c_s}{-\min \tilde{\omega}} \in (0,\infty)\,.
    \end{equation}
\end{itemize}
In contrast to the LIC cases, we saw in \cite{CS22} that for 2D anisotropic potentials with $0<s<1$ (for which $\tilde{\omega}$ is always sign-changing) the minimizers tend to collapse on 1D distributions for large $\alpha$. As an analogue, in Section \ref{sec_lowerdim} we exploit the cases of $W_\alpha$ for which the minimizers collapse on lower dimensional distributions for large $\alpha$. Here in the 3D case, the collapse phenomenon is richer than that of the 2D case because minimizers may collapse to 1D or 2D distributions, depending on how singular the potential is and how the function $\omega$ achieves its minimum on $S^2$. For such collapse to happen, one necessary condition is that $s$ is not too large, so that the energy is finite for such concentrated measures. Therefore, it is not surprising for us to obtain the following results, in the case of axisymmetric potentials with certain nondegeneracy conditions:
\begin{itemize}
    \item (Theorem \ref{thm_collapse1}) If $0<s<1$ and $\omega$ is minimized at $\theta=0$, then for sufficiently large $\alpha$, the energy minimizer for $W_\alpha$ is unique, given by $\rho_{\textnormal{1D}}$ (c.f. Appendix \ref{app:constants}).
    \item (Theorem \ref{thm_collapse2}) If $0<s<2$ and $\omega$ is minimized at $\theta=\pi/2$, then for sufficiently large $\alpha$, the energy minimizer for $W_\alpha$ is unique, given by $\rho_{\textnormal{2D}}$.
\end{itemize}
The proof is based on comparison arguments with specially designed potentials, similar to \cite{CS22}. We remark that the case when $\omega$ is minimized at some $\theta=\theta_0\in (0,\pi/2)$ is still open.

This part of our result answers an open problem proposed in \cite{CMMRSV2}, which is concerned with the energy minimizer for $W_\alpha$ with $s=1$, $\omega(\bbx) = \cos^2\theta$ and large $\alpha$. In fact, this $\omega$ clearly satisfies the assumptions of Theorem \ref{thm_collapse2}, and we may conclude that the unique energy minimizer for $W_\alpha$ is $\rho_{\textnormal{2D}}$ for sufficiently large $\alpha$.

Finally, in Section \ref{sec_expand} we study the case $1\le s < 2$ and $\omega$ minimized at $\theta=0$, which is in general not covered by the previously stated results. We first notice that the sign of $\tilde{\omega}$ is not completely determined (Theorem \ref{thm_s12exist}): for any fixed $1\le s<2$,
\begin{itemize}
    \item There exists axisymmetric $\omega$, minimized at $\theta=0$, such that $\tilde{\omega}$ is nonnegative (and thus $W_\alpha$ is always LIC and its minimizer is always ellipsoid-shaped).
    \item There also exists axisymmetric $\omega$, minimized at $\theta=0$, such that $\tilde{\omega}$ is sign-changing.
\end{itemize}
In either case, since $\omega$ is minimized at $\theta=0$, one might expect that the minimizers would elongate along the $x_3$-direction; also, since the power index $1\le s < 2$ does not allow the concentration to 1D measures, it would elongate to infinity. However, it turns out that this intuition is far from the truth. We will study the expansion of the minimizers from the following three aspects (for $1\le s<2$):
\begin{itemize}
    \item (Theorem \ref{thm_abalpha1}) If $\tilde{\omega}$ is strictly positive, then the ellipsoid-shaped minimizer for $W_\alpha$ necessarily expands to infinity \emph{in all dimensions} as $\alpha\rightarrow\infty$, with the ratio between its axes converging to a positive constant. This result also works in the case $2\le s < 3$. The proof is based on detailed analysis of the formula \eqref{lem_ab_2} which determines the axis lengths for the ellipsoids via $A=B=1$. 
    \item (Theorem \ref{thm_abalpha2}) There exists $\omega$, minimized at $\theta=0$ with $\tilde{\omega}\ge 0$, such that the ellipsoid-shaped minimizer for $W_\alpha$ expands in $x_1$ and $x_2$-directions but \emph{shrinks} in $x_3$-direction as $\alpha\rightarrow\infty$. It is proved by an explicit construction with an argument similar to the previous result.
    \item (Theorem \ref{thm_cylinder}) For general axisymmetric $\omega$, as long as it is positive at $\theta=\pi/2$, the minimizers for $W_\alpha$ have to expand in \emph{at least two dimensions} as $\alpha\rightarrow\infty$, i.e., these minimizers cannot be supported inside a fixed infinite cylinder for all $\alpha$. The proof is based on a comparison argument with isotropic energies, together with a scaling analysis for $\alpha$.
\end{itemize}
These results show that the condition that $\omega$ minimized at $\theta=0$ does not imply that $x_3$-direction is preferred by the minimizers for $W_\alpha$. As $\alpha$ gets large, the minimizers may or may not expand in $x_3$-direction, and they have to expand in at least two dimensions. 

We remark that the behavior of the energy minimizers for $W_\alpha$ with intermediate $\alpha$ remains largely open. This is also the case for large $\alpha$ for  $1\le s < 2$ with $\omega$ minimized at $\theta=\pi/2$ and $\tilde{\omega}$ sign-changing, for which the only thing we know is its expansion phenomenon in Theorem \ref{thm_cylinder}. In these cases the energy is not LIC, and the collapse to lower dimensions happens. In fact, we know that the potential is infinitesimal concave (Proposition \ref{prop_concave}) and thus one expects the minimizers to be supported in lower dimensional sets and/or fractals. Also, one can conduct local analysis for the generated potential, and show that 1D/2D fragments along certain directions are prohibited in any Wasserstein-$\infty$ local minimizers, as was done in Section 6 of \cite{CS22}. We expect this behavior to be related to sign information of $\tilde{\Omega}$. 

Finally, we point out that the logarithmic case corresponding formally to $s=0$ can also be included in 3D following a similar limiting procedure as in \cite[Section 7]{CS22}.

\section{Fourier transform and the LIC property}\label{sec_LIC}

\subsection{Fourier transform}

We first give the formula for the Fourier transform for functions of the form $|\bx|^{-s}\Omega(\bbx)$. We will state it for complex numbers $s$ in the region
\begin{equation}\label{kS}
    s\in \kS := \{s\in\mathbb{C}:0<\Re(s)<3\}
\end{equation}
and analyze its holomorphic property. For any function $\Omega$ defined on $S^2$ and $\bxi\in S^2$, we define
\begin{equation}
    [\Omega]_{\bxi} := \frac{1}{2\pi}\int_{\bbx\cdot\bxi=0} \Omega(\bbx)\rd{\bbx}
\end{equation}
where the integral is with respect to the induced measure on the 1D submanifold $\{\bbx\in S^2:\bbx\cdot\bxi=0\}$. It is the average of $\Omega$ on this submanifold.

\begin{lemma}\label{lem_FT}
For $\Omega$ satisfying {\bf (H0)} and any complex number $s\in \kS$ (as defined in \eqref{kS}), we have
\begin{equation}\label{lem_FT_1}
    \cF[|\bx|^{-s}\Omega(\bbx)] = |\xi|^{-3+s}\tilde{\Omega}(\bxi;s)
\end{equation}
for some function $\tilde{\Omega}(\bxi;s)$. The function $\tilde{\Omega}(\bxi;s)$ is smooth in $\bxi$ and holomorphic in $s$. Furthermore, $\partial_s\tilde{\Omega}(\bxi;s)$ is also smooth in $\bxi$. $\tilde{\Omega}(\bxi;s)$ is given by the formulas for particular real values of $s$ (omitting $s$-dependence when unnecessary):
\begin{equation}\label{lem_FT_2}
    \tilde{\Omega}(\bxi) = \tau_{3-s}\int_{S^2}|\bbx\cdot\bxi|^{-3+s}\Omega(\bbx) \rd{\bbx},\quad 2<s<3
\end{equation}
where $\tau_{3-s}$ is defined in \eqref{calc4}, and 
\begin{equation}\label{lem_FT_2b}
    \tilde{\Omega}(\bxi) = \pi[\Omega]_{\bxi},\quad s=2.
\end{equation}
\end{lemma}

\begin{remark}
We may view $\tilde{\Omega}(\bxi;s)$ as the analytic continuation of \eqref{lem_FT_2} from $2<s<3$ to the region $s\in \kS$. The formulas corresponding to \eqref{lem_FT_1} for the range $1<s<2$ can also be obtained after additional care of the singularities. However, these formulas will never be used in the rest of this work, and thus we have postponed them to the Appendix \ref{app_FT}. Moreover, applying \eqref{lem_FT_2} reversely, we get
\begin{equation}\label{lem_FT_4}
    \Omega(\bbx) = \tau_{s}\int_{S^2}|\bbx\cdot\bxi|^{-s}\tilde{\Omega}(\bxi) \rd{\bxi},\quad 0<s<1,
\end{equation}
which gives the decomposition of the potential into a linear combination of 1D potentials as
\begin{equation}\label{decomp}
    |\bx|^{-s}\Omega(\bbx) = \tau_{s}\int_{S^2}|\bx\cdot\bxi|^{-s}\tilde{\Omega}(\bxi) \rd{\bxi},\quad 0<s<1.
\end{equation}
Similarly, from \eqref{lem_FT_2b}, we get
\begin{equation}\label{lem_FT_4b}
    \Omega(\bbx) = \pi[\tilde{\Omega}]_{\bbx},\quad s=1.
\end{equation}
\end{remark}

\begin{proof}
We first show that $\cF[|\bx|^{-s}\Omega(\bbx)]$ is a locally integrable function for any complex number $s\in \kS$. We fix the Littlewood-Paley cutoff function $\psi(\bx)$, which is radial, smooth, nonnegative, supported on $\{1\le |\bx|\le 4\}$ and $\psi(\bx)+\psi(\bx/2)=1$ on $\{2\le |\bx|\le 4\}$. Then we may decompose $|\bx|^{-s}\Omega(\bbx)$ as
\begin{equation}\label{FTdecomp}
    |\bx|^{-s}\Omega(\bbx) = \Big(1-\sum_{k=0}^\infty\psi(2^{-k}\bx)\Big)|\bx|^{-s}\Omega(\bbx) + \sum_{k=0}^\infty\psi(2^{-k}\bx)|\bx|^{-s}\Omega(\bbx) =: I_1+I_2,
\end{equation}
where the last summation $I_2$ converges in the sense of distributions. Since $\Re(s)\in (0,3)$, $I_1$ is locally integrable. $I_1$ is supported on $\{|\bx|\le 2\}$, and thus $\cF[I_1](\xi)$ is in $L^\infty$ and smooth in $\xi$ (for every fixed $s\in \kS$). Also, one can differentiate $\cF[I_1](\xi)$ with respect to $s$ by 
\begin{equation}
    \partial_s\cF[I_1](\xi) = \int_{\mathbb{R}^3} \Big(1-\sum_{k=0}^\infty\psi(2^{-k}\bx)\Big)|\bx|^{-s}(-\ln|\bx|)\Omega(\bbx)e^{-2\pi i \bx\cdot\xi}\rd{\bx}
\end{equation}
since the last integral converges absolutely. This shows $\cF[I_1](\xi)$ is holomorphic in $s$ and $\partial_s\cF[I_1](\xi)$ is smooth in $\xi$.

Each term $\psi(2^{-k}\bx)|\bx|^{-s}\Omega(\bbx)$ in the summation $I_2$ in \eqref{FTdecomp} can be written as $2^{-k s}g(2^{-k}\bx)$, where $g(\bx) = \psi(\bx)|\bx|^{-s}\Omega(\bbx)$ is a compactly supported smooth function. Therefore its Fourier transform can be computed as
\begin{equation}
    \cF[I_2](\xi) =  \sum_{k=0}^\infty 2^{k (3-s)}\hat{g}(2^k\xi),
\end{equation}
where the last summation converges in the sense of distribution, and also pointwisely for every $\xi\ne 0$ since $\hat{g}$ is a Schwartz function. This summation also converges in $L^1$ since $\|2^{k (3-s)}\hat{g}(2^k\xi)\|_{L^1} = 2^{-k\Re(s)}\|\hat{g}\|_{L^1} $ which is summable in $k$. As a consequence, $\cF[I_2]$ is in $L^1$.

It is also clear that $\cF[I_2](\xi)$ is smooth in $\xi$ for $\xi\ne 0$. Furthermore, for every $\xi\ne 0$, one can differentiate the above summation with respect to $s$ and obtain
\begin{equation}
    \partial_s \cF[I_2](\xi) =  \sum_{k=0}^\infty 2^{k (3-s)}(-k\ln 2)\hat{g}(2^k\xi).
\end{equation}
Therefore, $\cF[I_2](\xi)$ is holomorphic in $s$ for any $\xi\ne 0$. Furthermore, this expression shows that $\partial_s \cF[I_2](\xi)$ is also smooth in $\xi$ for $\xi\ne 0$.

Combining the above results, we see that for $s\in \kS$, $\cF[|\bx|^{-s}\Omega(\bbx)]$ is in $L^1+L^\infty$, smooth in $\xi$ for any $\xi\ne 0$, and holomorphic in $s$ for any fixed $\xi\ne 0$. Scaling argument (by replacing $\bx$ with $\lambda \bx,\,\lambda>0$) shows that $\cF[|\bx|^{-s}\Omega(\bbx)]$ has to take the form \eqref{lem_FT_1} for some function $\tilde{\Omega}$, and then we see that $\tilde{\Omega}(\bxi;s)$ is smooth in $\bxi$ and holomorphic in $s$, and $\partial_s\tilde{\Omega}(\bxi;s)$ is also smooth in $\bxi$.

To prove \eqref{lem_FT_2}, we fix $2<s<3$, take any $\xi\ne 0$, and calculate $\cF[|\bx|^{-s}\Omega(\bbx)](\xi)$ as an improper integral
\begin{equation}\begin{split}
    \cF[|\bx|^{-s}\Omega(\bbx)](\xi) = & \lim_{R\rightarrow\infty}\int_{\cB(0;R)} |\bx|^{-s}\Omega(\bbx)e^{-2\pi i \bx\cdot \xi}\rd{\bx} \\
    = & \lim_{R\rightarrow\infty}\int_{S^2}\int_0^R \cos(2\pi r\bbx\cdot\xi) r^{-s+2}\rd{r}\Omega(\bbx) \rd{\bbx} \\
    = & \lim_{R\rightarrow\infty}\int_{S^2}\int_0^{R/|2\pi \bbx\cdot\xi|}  r^{-s+2}\cos r\rd{r}|2\pi \bbx\cdot\xi|^{-3+s}\Omega(\bbx) \rd{\bbx} \\
    = & \int_{S^2}\int_0^\infty  r^{-s+2}\cos r\rd{r}|2\pi \bbx\cdot\xi|^{-3+s}\Omega(\bbx) \rd{\bbx} \\
\end{split}\end{equation}
where we use the fact that $\cF[|\bx|^{-s}\Omega(\bbx)]$ is real (since $|\bx|^{-s}\Omega(\bbx)$ is even) in the second equality, and dominated convergence theorem in the last equality. Formula \eqref{calc3} gives the value of the improper integral $\int_0^\infty  r^{-s+2}\cos r\rd{r} = -\Gamma(3-s)\sin\frac{(-s+2)\pi}{2}$. Therefore we obtain \eqref{lem_FT_2}.

\eqref{lem_FT_2b} can be derived by taking the limit $s\rightarrow 2^+$ in \eqref{lem_FT_2}. In fact, by rotational symmetry, we may assume $\bxi=(0,0,1)$ without loss of generality. Since $\tilde{\Omega}((0,0,1);s)$ is holomorphic in $s$, we have
\begin{equation}\begin{split}
    \tilde{\Omega}((0,0,1);2) = & \lim_{s\rightarrow 2^+}\tilde{\Omega}((0,0,1);s) = \lim_{s\rightarrow 2^+}\tau_{3-s}\int_{S^2}|\bbx\cdot(0,0,1)|^{-3+s}\Omega(\bbx) \rd{\bbx} \\
    = & \lim_{s\rightarrow 2^+}\tau_{3-s}\int_0^\pi \int_0^{2\pi}\Omega(\bbx)\rd{\mu} |\cos\theta|^{-3+s}\sin\theta\rd{\theta}
\end{split}\end{equation}
Notice that $\tau_{3-s}=(2\pi)^{-3+s}\Gamma(3-s)\cos\frac{(3-s)\pi}{2}$ behaves like $(2\pi)^{-1}\frac{(s-2)\pi}{2} = \frac{s-2}{4}$ as $s\rightarrow 2^+$, and $\int_0^\pi|\cos\theta|^{-3+s}\sin\theta\rd{\theta}=\frac{2}{s-2}$. Since $|\cos\theta|^{-3+s}\sin\theta$ concentrates near $\theta=\pi/2$ as $s\rightarrow 2^+$, we see that $2\tau_{3-s}|\cos\theta|^{-3+s}\sin\theta$ forms an approximation of identity in $\theta\in [0,\pi]$. Therefore we obtain
\begin{equation}\begin{split}
    \tilde{\Omega}((0,0,1);2) =  \frac{1}{2}\int_0^{2\pi}\Omega(\bbx|_{\theta=\pi/2})\rd{\mu} = \pi [\Omega]_{(0,0,1)}
\end{split}\end{equation}
as desired.

\end{proof}

For $1<s<3$, we give another formula for the Fourier transform.

\begin{lemma}\label{lem_FTpsi}
For $\Omega$ satisfying {\bf (H0)} and $1<s<3$, if $\Omega$ is given by
\begin{equation}\label{omega_psi}
    \Omega(\bbx) = \int_{S^2}\delta(\bbx\cdot\bby)\psi(\bby)\rd{\bby}
\end{equation}
for some smooth function $\psi$ defined on $S^2$, then the Fourier transform of $|\bx|^{-s}\Omega(\bbx)$ is given by \eqref{lem_FT_1} with
\begin{equation}\label{tomega_psi}
    \tilde{\Omega}(\bxi) = c_{s-1,\textnormal{2D}}\int_{S^2}(1-|\bby\cdot\bxi|^2)^{(-3+s)/2}\psi(\bby)\rd{\bby}
\end{equation}
\end{lemma}
Here, for $0<s<2$, the constant $c_{s,\textnormal{2D}}=\pi^{s-1}\frac{\Gamma((2-s)/2)}{\Gamma(s/2)}>0$ refers to the `$c_s$' constant in 2D as in \cite{CS22}. The $\delta$ in \eqref{omega_psi} refers to the 1D Dirac delta function.

It is clear that if $\psi$ is axisymmetric, so are $\Omega$ and $\tilde{\Omega}$.

\begin{proof}
We first write the function $|\bx|^{-s}\Omega(\bbx)$ as
\begin{equation}
    |\bx|^{-s}\Omega(\bbx) = \int_{S^2}|\bx|^{-s}\delta(\bbx\cdot\bby)\psi(\bby)\rd{\bby}
\end{equation}

For $\bby=(0,0,1)$, we have\footnote{Here we use the rescaling rule $\delta(\lambda t)=\lambda^{-1}\delta(t),\,\lambda>0$ for the 1D Dirac delta function.}
\begin{equation}
    |\bx|^{-s}\delta(\bbx\cdot\bby)=|\bx|^{-s}\delta\Big(\frac{x_3}{|\bx|}\Big)=|x_1^2+x_2^2|^{(1-s)/2}\delta(x_3)
\end{equation}
whose Fourier transform is given by $c_{s-1,\textnormal{2D}}|\xi_1^2+\xi_2^2|^{(-3+s)/2}$ as a function in $\xi\in\mathbb{R}^3$, for any $1< s < 3$. We may write $|\xi_1^2+\xi_2^2|^{(-3+s)/2}=(|\xi|^2-|\bby\cdot\xi|^2)^{(-3+s)/2}$ for $\bby=(0,0,1)$. Therefore, applying this with suitable rotation and integrating in $\bby$, we obtain
\begin{equation}
    \cF[|\bx|^{-s}\Omega(\bbx)] = c_{s-1,\textnormal{2D}}\int_{S^2}(|\xi|^2-|\bby\cdot\xi|^2)^{(-3+s)/2}\psi(\bby)\rd{\bby} = |\xi|^{-3+s}\tilde{\Omega}(\bxi)
\end{equation}
with $\tilde{\Omega}$ given by \eqref{tomega_psi}.
\end{proof}

For the purpose of later applications, we take $\psi$ as the rescalings of a fixed mollifier in the previous lemma and analyze the behavior of $\Omega$ and $\tilde{\Omega}$.
\begin{lemma}\label{lem_FTpsi2}
Assume $1<s<3$. Take a fixed nonnegative smooth even function $\psi_1(\theta)$ supported on $\theta\in[-1,1]$, and define $\psi_\epsilon(\theta)=\frac{1}{\epsilon^2}\psi_1(\frac{\theta}{\epsilon})$ for small $\epsilon>0$. View $\psi_\epsilon$ as an axisymmetric function on $S^2$, and denote the resulting $\Omega$ in \eqref{omega_psi} as $\Omega^\epsilon$. Then
\begin{itemize}
    \item $\Omega^\epsilon$ satisfies {\bf (H0)}, is nonnegative, axisymmetric, and supported on $\theta\in [\pi/2-\epsilon,\pi/2+\epsilon]$.
    \item $\|\Omega^\epsilon\|_{L^\infty} \sim \epsilon^{-1}$, with $\Omega^\epsilon(\theta) \sim \epsilon^{-1}$ for $\theta\in [\pi/2-\epsilon/2,\pi/2+\epsilon/2]$.
    \item $\|\tilde{\Omega}^\epsilon\|_{L^\infty} \sim \epsilon^{-3+s}$, with $\tilde{\Omega}^\epsilon(\varphi) \sim \epsilon^{-3+s}$ for $\varphi\in [0,\epsilon/2]$.
\end{itemize}
Here $\sim$ means bounded above and below by positive constants. Furthermore, for properly chosen $\psi_1$ (to be specified in the proof) and sufficiently small $\epsilon$, we have
\begin{itemize}
    \item $\tilde{\Omega}^\epsilon$, as a function of $\varphi$, is decreasing in $\varphi\in [0,\pi/2]$.
    \item There exist positive constants $c_{\psi,1},c_{\psi,2}$, independent of $\epsilon$, such that
    \begin{equation}
        \tilde{\Omega}^\epsilon(\varphi) \le \Big(1-c_{\psi,1}\frac{\varphi^2}{\epsilon^2}\Big)\tilde{\Omega}^\epsilon(0),\quad \forall \varphi\in [0,c_{\psi,2}\epsilon]
    \end{equation}
    
\end{itemize}

\end{lemma}

See Figure \ref{fig:lem_FTpsi} for illustration.
\begin{figure}
    \centering
    \includegraphics[width=0.49\textwidth]{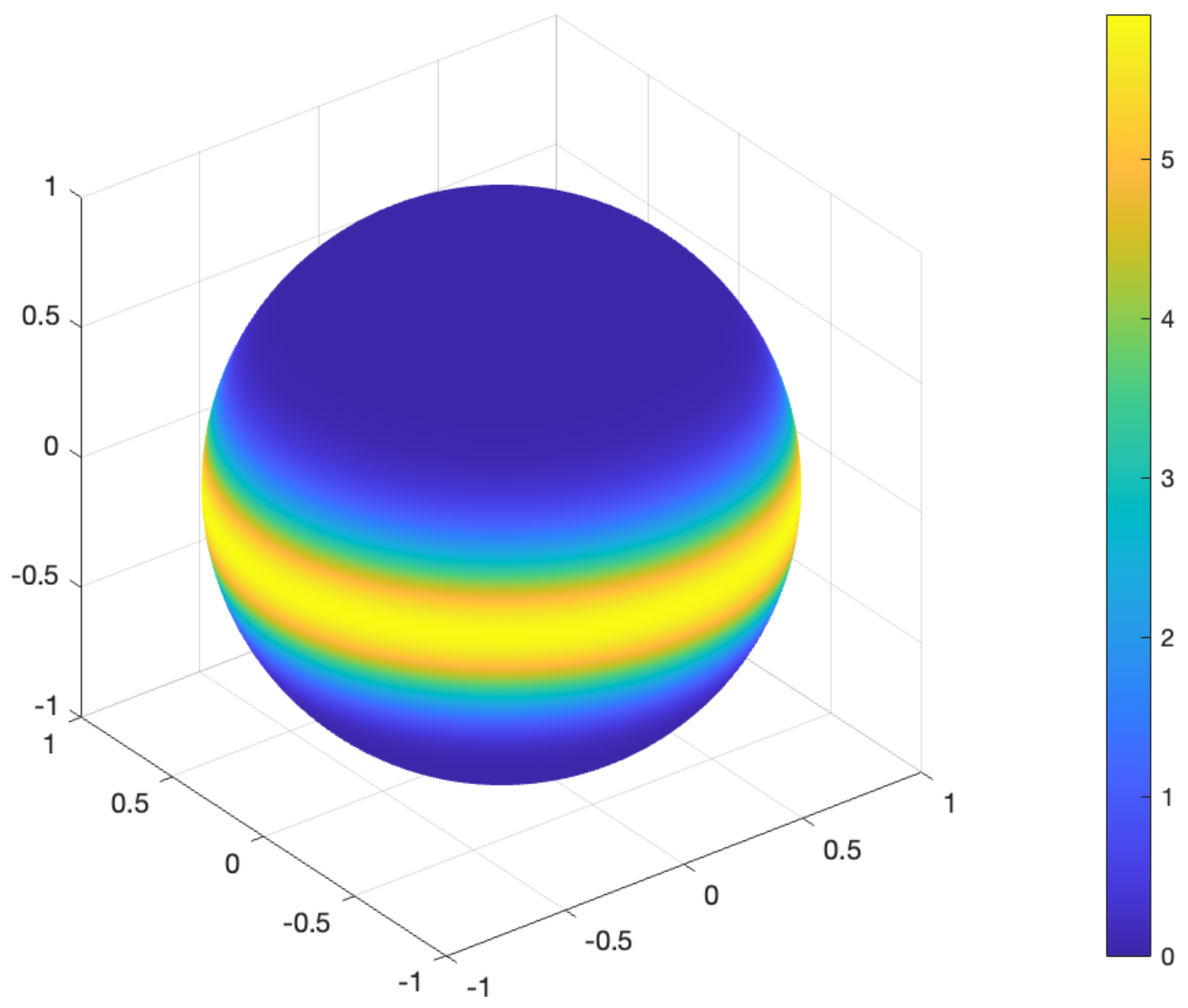}
    \includegraphics[width=0.49\textwidth]{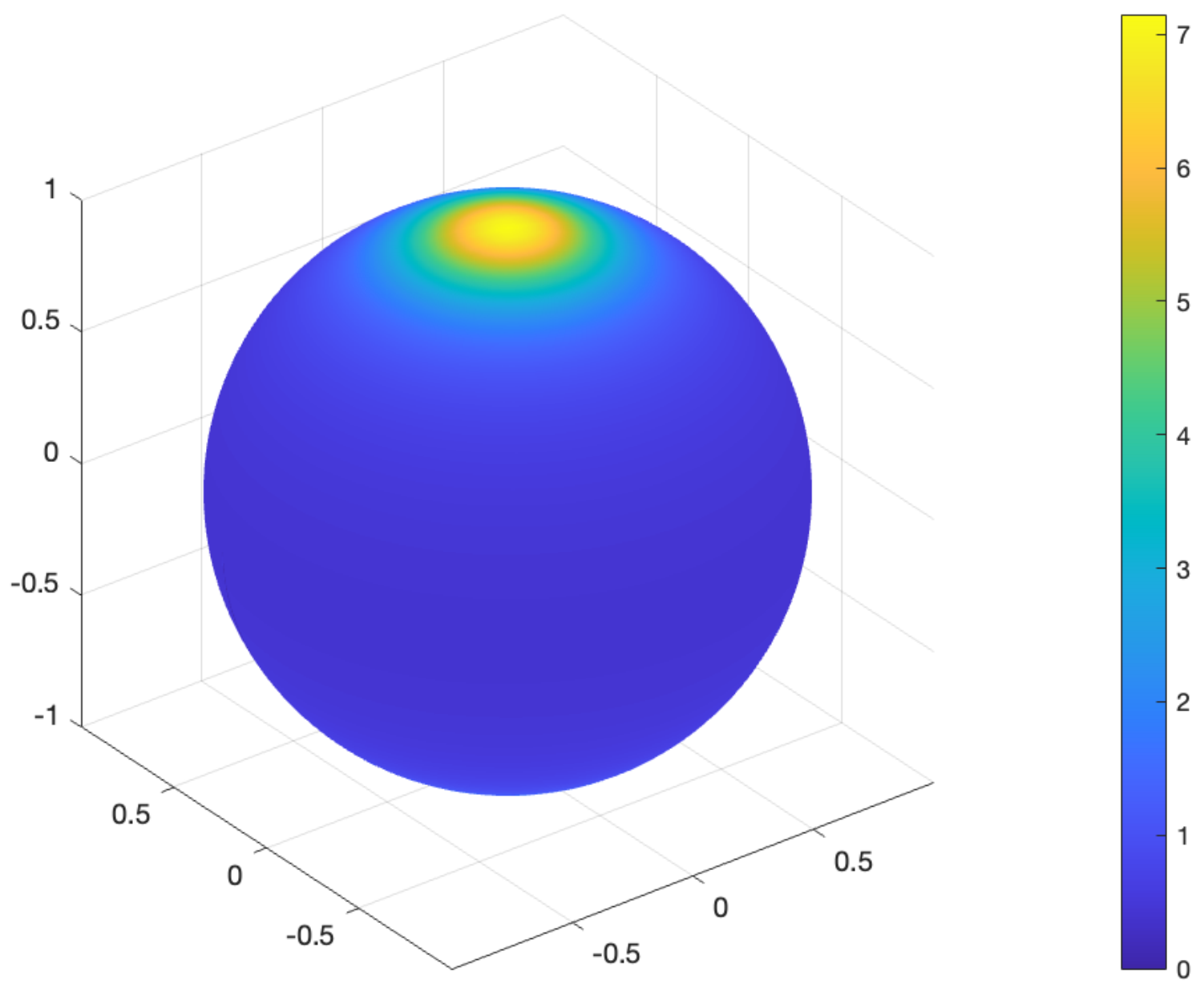}
    \caption{Illustration of Lemma \ref{lem_FTpsi2}. Left: $\Omega^\epsilon$ is concentrated near $\theta=\pi/2$; right: $\tilde{\Omega}^\epsilon$ is concentrated near $\varphi=0$.}
    \label{fig:lem_FTpsi}
\end{figure}

\begin{proof}
Item 1 is clear. To see item 2, it suffices to notice that $\Omega^\epsilon(\bbx)$ is an average of the spherical function $\delta((\cdot)\cdot(0,0,1))$ on a ball of radius $\epsilon$ centered at $\bbx$, up to a negligible curvature effect from the sphere. To see item 3, we notice that $\tilde{\Omega}^\epsilon(\bxi)$ is an average of the spherical function $c(1-|(\cdot)\cdot(0,0,1)|^2)^{(-3+s)/2}$ on a ball of radius $\epsilon$ centered at $\bxi$. The last function has a singularity $c(1-\xi_3^2)^{(-3+s)/2} \sim c (\sin\varphi)^{-3+s}$ near $\varphi=0$. Therefore item 3 follows.

To prove items 4 and 5, we take a specific $\psi_1$ as
\begin{equation}
    \psi_1(\theta) = \exp\left(-\frac{1}{1-\theta^2}\right)\chi_{(-1,1)}(\theta).
\end{equation}
It is clear that $\psi_1(\theta)$ is strictly decreasing on $\theta\in [0,1]$. Furthermore, we claim that\footnote{Here $\Delta_{S^2}$ denotes the Laplace-Beltrami operator on $S^2$.} $\Delta_{S^2}\psi_\epsilon(\theta)$ strictly negative on $[0,\tau_{\epsilon}\epsilon)$ and strictly positive on $(\tau_{\epsilon}\epsilon,\epsilon)$ for any $\epsilon>0$ sufficiently small and some $\tau_{\epsilon}\in (0.1,0.9)$. In fact, explicit calculation shows that for any $\theta\in (0,1)$,
\begin{equation}
    \partial_\theta\psi_1(\theta) = \psi_1(\theta)\frac{-2\theta}{(1-\theta^2)^2},\quad \partial_{\theta\theta}\psi_1(\theta) = \psi_1(\theta)\frac{6\theta^4-2}{(1-\theta^2)^4}.
\end{equation}
Therefore, we deduce
\begin{equation}\begin{split}
    \sin\theta\Delta_{S^2}\psi_\epsilon(\theta) = & \partial_\theta(\sin\theta \partial_\theta\psi_\epsilon(\theta)) = \cos\theta \partial_\theta\psi_\epsilon(\theta) + \sin\theta \partial_{\theta\theta}\psi_\epsilon(\theta) \\
    = & \frac{1}{\epsilon^3}\cos\theta (\partial_\theta\psi_1)(\frac{\theta}{\epsilon}) + \frac{1}{\epsilon^4}\sin\theta (\partial_{\theta\theta}\psi_1)(\frac{\theta}{\epsilon}) \\
    = & \frac{1}{\epsilon^4}\psi_1(\frac{\theta}{\epsilon})\left(\epsilon\cos\theta \frac{-2(\theta/\epsilon)}{(1-(\theta/\epsilon)^2)^2} + \sin\theta \frac{6(\theta/\epsilon)^4-2}{(1-(\theta/\epsilon)^2)^4} \right).
\end{split}\end{equation}
We further compute
\begin{align*}
    \sin\epsilon\theta\Delta_{S^2}\psi_\epsilon(\epsilon\theta) 
    = & \frac{\theta\cos\epsilon\theta}{\epsilon^3}\psi_1(\theta)\left(\frac{-2}{(1-\theta^2)^2} + \frac{\tan\epsilon\theta}{\epsilon\theta} \frac{6\theta^4-2}{(1-\theta^2)^4} \right) \\
    = & \frac{\theta\cos\epsilon\theta}{\epsilon^3(1-\theta^2)^4}\psi_1(\theta)\Big(\!\!-2(1-\theta^2)^2 + \frac{\tan\epsilon\theta}{\epsilon\theta} (6\theta^4-2)\Big) \\
    = & \frac{\theta\cos\epsilon\theta}{\epsilon^3(1-\theta^2)^4}\psi_1(\theta)\Big[\Big(6\frac{\tan\epsilon\theta}{\epsilon\theta}-2\Big)\theta^4 + 4\theta^2 + \Big(\!\!-2 -2\frac{\tan\epsilon\theta}{\epsilon\theta}\Big)\Big]. 
\end{align*}
Using the fact that $\frac{\tan\epsilon\theta}{\epsilon\theta}$ is close to 1 for small $\epsilon>0$ and $\theta\in (0,1)$, one can show that the last bracket is well-approximated by $4\theta^4 + 4\theta^2 - 4$. In fact, as $\epsilon\rightarrow 0^+$, this function and its derivative with respect to $\theta^2$ converge to those of the polynomial uniformly on $[0,1]$. This polynomial has positive derivative (with respect to $\theta^2$) and changes sign once in $\theta\in (0,1)$, at $\theta_*=((\sqrt{5}-1)/2)^{1/2}\approx 0.786$. Therefore the last bracket, and thus $\Delta_{S^2}\psi_\epsilon(\epsilon\theta)$, only changes sign once in $\theta\in (0,1)$ close to $\theta_*$, which proves the claim.

We notice a basic fact that the function $\varphi\mapsto \int_{S^2}g_1(\bby\cdot(\sin\varphi,0,\cos\varphi))g_2(\bby\cdot(0,0,1))\rd{\bby}$ is decreasing in $\varphi\in [0,\pi/2]$ provided that $g_i(t),\,i=1,2$, are nonnegative even functions on $[-1,1]$ and increasing on $[0,1]$. To see this, it suffices to check it for $g_i(t) = \chi_{[a_i,1]}(|t|)$ with $a_1,a_2\in [0,1]$, for which the spherical integral can be calculated explicitly.

Then item 4 follows from \eqref{tomega_psi} and the fact that the input $\psi_\epsilon$ is decreasing in $\theta$ by construction, due to the application of the previous fact, with $g_1(t)=(1-t^2)^{(-3+s)/2}$ and $g_2(\cos\theta)=\psi_\epsilon(\theta)$.

To see item 5, we notice from \eqref{tomega_psi} that (as a function of $\varphi$)
\begin{equation}
    \Delta_{S^2}\tilde{\Omega}^\epsilon(0) =  c_{s-1,\textnormal{2D}}\int_{S^2}(1-|\bbx\cdot(0,0,1)|^2)^{(-3+s)/2}\Delta_{S^2}\psi_\epsilon (\bbx)\rd{\bbx}
\end{equation}
since the integral operator on $S^2$ given by the integral kernel $(1-|\bbx\cdot\bxi|^2)^{(-3+s)/2}$ commutes with $\Delta_{S^2}$. By our choice of $\psi_\epsilon$, the function $\Delta_{S^2}\psi_\epsilon (\bbx)$ is axisymmetric, strictly negative on $\theta\in [0,\tau_{\epsilon}\epsilon)$, strictly positive on $\theta\in (\tau_{\epsilon}\epsilon,\epsilon)$. It is also mean-zero on $\{\bbx\in S^2:\theta\le \epsilon\}$ because $\psi_\epsilon$ is compactly-supported on this set. The function $(1-|\bbx\cdot(0,0,1)|^2)^{(-3+s)/2}$ is axisymmetric, positive and decreasing in $\theta$ for $\theta\in (0,\pi/2]$. Therefore we see that $\Delta_{S^2}\tilde{\Omega}^\epsilon(0)<0$. In fact, by analyzing the scaling that $(1-|\bbx\cdot(0,0,1)|^2)^{(-3+s)/2}\sim \theta^{-3+s}$ and $|\Delta_{S^2}\psi_\epsilon (\bbx)|\sim \epsilon^{-4}$ for $\theta \le c\epsilon$, one can quantify it as
\begin{equation}
    \Delta_{S^2}\tilde{\Omega}^\epsilon(0) \le - c \epsilon^{-5+s}
\end{equation}
for sufficiently small $\epsilon$.

Combined with a similar scaling argument for the spherical gradient of $\Delta_{S^2}\tilde{\Omega}^\epsilon$, one can show that the above inequality is also true for $\Delta_{S^2}\tilde{\Omega}^\epsilon(\varphi)$ with $\varphi \le c\epsilon$. This gives
\begin{equation}
    \tilde{\Omega}^\epsilon(\varphi) \le \tilde{\Omega}^\epsilon(0) - c\epsilon^{-5+s}\varphi^2
\end{equation}
for $\varphi \le c\epsilon$. Combined with item 3, we get item 5.

\end{proof}

\subsection{Results on the LIC property}

In this subsection we state some results on the existence of energy minimizers and the LIC property of the potential $W$ given by \eqref{W}, as generalization of those in \cite{CS22} to 3D. The proofs are similar to those in \cite{CS22} and thus omitted.

\begin{lemma}\label{lem_exist}
Assume $0<s<3$. Then for any $W$ in \eqref{W} with $\Omega$ satisfying {\bf (H)}, there exists a compactly supported energy minimizer in the class of probability measures on $\mathbb{R}^3$. The same is true for $W_\alpha$ in \eqref{Walpha} with $\alpha\ge 0$ if $\omega$ satisfies {\bf (h)}. 

If we further assume $0<s<1$, then any minimizer for $W_\alpha$ with zero center of mass is supported in $\cB(0;R)$ for some $R>0$ independent of $\alpha$. The same is true if we instead assume $1\leq s<2$ and $\omega(\bbx)=0$ for any $\bbx$ with $\theta=\pi/2$.
\end{lemma}
In the second part of the above lemma, the extra conditions guarantee that $\min_\rho E_\alpha[\rho]$ is uniformly bounded in $\alpha$ (for $0<s<1$, a possibly rotated $\rho_{\textnormal{1D}}$ has energy independent of $\alpha$; for $1\leq s<2$ and $\omega|_{\theta=\pi/2}=0$, $\rho_{\textnormal{2D}}$ has energy independent of $\alpha$). This is crucial in the proof of the uniform-in-$\alpha$ bound for the support of minimizers, as done in  \cite[Lemma B.1]{CS22}.

The following lemma shows that for LIC potentials, an Euler-Lagrange condition for the energy minimizer is also sufficient.
\begin{lemma}\label{lem_EL}
Assume $W$ satisfies {\bf (W)}, and $W$ has the LIC property. Assume there exists a compactly supported global energy minimizer (which has to be unique up to translation). Then for any probability measure $\rho$ with $E[\rho]<\infty$, the following are equivalent:
\begin{itemize}
\item[(i)] $\rho$ is the unique energy minimizer for $W$ up to translation.
\item[(ii)] $\rho$ satisfies the condition
\begin{equation}\left\{\begin{split}
    & (W*\rho)(\bx) = 2E[\rho],\quad \rho\textnormal{ a.e.}\\
    & (W*\rho)(\bx) \le  2E[\rho],\quad \forall \bx\in \supp\rho \\
    & (W*\rho)(\bx) \ge 2E[\rho],\quad\textnormal{ a.e.}\,\bx \\
\end{split}\right.\end{equation}
\item[(iii)] $\rho$ satisfies the condition
\begin{equation}\label{EL}
    (W*\rho)(\bx) \le \essinf(W*\rho),\quad \rho \textnormal{ a.e.}
\end{equation}
\end{itemize}
\end{lemma}
\begin{proof}
Here, (i)$\Rightarrow$(ii) is \cite[Theorem 4]{BCLR2} which is true for any $W$ satisfying {\bf (W)}; (ii)$\Rightarrow$(iii) is clear for any $W$ satisfying {\bf (W)}; (iii)$\Rightarrow$(i) can be proved by using the LIC condition with a mollification argument, similar to the proof of \cite[Lemma 2.2]{carrilloshu21}.
\end{proof}

A direct application of the results in \cite{carrilloshu21} as fully detailed in \cite[Section 2]{CS22} leads to the following consequences.

\begin{proposition}\label{prop_concave}
Let $W$ be given by \eqref{W} with $0<s<3$ and $\Omega$ satisfying {\bf (H)}. Assume
$\tilde{\Omega}$ is negative somewhere. Then $W$ is infinitesimally concave, i.e., for any $\epsilon>0$, there exists a function $\mu\in L^\infty(\mathbb{R}^3)$ such that $\int \mu = 0$, $\supp\mu \subset \cB(0;\epsilon)$ and $E[\mu]<0$. As a result, any superlevet set of any Wasserstein-$\infty$ local minimizer does not have interior points.
\end{proposition}

\begin{theorem}
\label{thm_LICequiv}
Let $W$ be given by \eqref{W} with $0<s<3$ and $\Omega$ satisfying {\bf (H)}. Then $W$ has the LIC property if and only if $\tilde{\Omega}$ given as in \eqref{lem_FT_1} is nonnegative.
\end{theorem}

For the class of potentials $W_\alpha$ in \eqref{Walpha} with $\omega$ satisfying {\bf (h)}, the corresponding angle function for the Fourier transform of $|\bx|^{-s}(1+\alpha\omega(\bbx))$ is $\tilde{\Omega}_\alpha = c_s + \alpha \tilde{\omega}$, where $c_s$ is given in \eqref{calc2}. According to the sign of $\tilde{\omega}$, the behavior of its energy minimizers can be categorized as:
\begin{itemize}
    \item If $\tilde{\omega}$ is nonnegative, then $W_\alpha$ has the LIC property for any $\alpha\ge 0$. $\tilde{\omega}$ is necessarily nonnegative if $2\le s < 3$.
    \item If $\tilde{\omega}$ is sign-changing, then $W_\alpha$ has the LIC property if $0\le \alpha \le \alpha_L$ where $\alpha_L$ is defined in \eqref{alphaL}. If $\alpha>\alpha_L$, then $W_\alpha$ does not have the LIC property, and thus infinitesimal concave by Proposition \ref{prop_concave}. $\tilde{\omega}$ is necessarily sign-changing if $0<s<1$.
\end{itemize}
We will see in Theorem \ref{thm_s12exist} that both cases can happen if $1\le s < 2$. Notice that this is the most novel case compared to the two dimensional results in \cite{CS22}.

\section{Ellipsoid-shaped minimizers for LIC potentials}\label{sec_ell}

For simplicity, we will focus on the potentials $W$ in \eqref{W} with the extra assumption
\begin{equation}
    \text{{\bf (Hx)}: $\Omega$ satisfies {\bf (H)} and axisymmetric with respect to the $x_3$-axis.}
\end{equation}
In other words, $\Omega$ is a function of $\theta\in [0,\pi]$ with $\Omega(\theta)=\Omega(\pi-\theta)$ (abusing notation, denoting $\Omega(\bbx)=\Omega(\theta)$). Its Fourier transform is also axisymmetric with respect to the $\xi_3$-axis, and we may write $\tilde{\Omega}(\bxi)=\tilde{\Omega}(\varphi)$. For $\omega$, we also introduce a similar assumption
\begin{equation}
    \text{{\bf (hx)}: $\omega$ satisfies {\bf (h)} and axisymmetric with respect to the $x_3$-axis.}
\end{equation}

For $a,b>0$, denote 
\begin{equation}\label{rhoab}
    \rho_{a,b}(\bx) = \frac{1}{a^2 b}\rho_3\Big(\frac{x_1}{a},\frac{x_2}{a},\frac{x_3}{b}\Big)
\end{equation}
as an axisymmetric rescaling of $\rho_3$ (defined in \eqref{rhod}).  $\rho_{0,b}$ is understood as the weak limit of $\rho_{a,b}$ as $a\rightarrow 0^+$, which is supported on the $x_3$-axis. Similarly $\rho_{a,0}$ is supported on the $x_1x_2$-plane. The support of $\rho_{a,b}$ with $(a,b)\in [0,\infty)^2\backslash \{(0,0)\}$ is a possibly degenerate ellipsoid with axis lengths $a R_3, a R_3, b R_3$ in the $x_1,x_2,x_3$ directions respectively.

\begin{theorem}\label{thm_ell}
Assume $0<s<3$, $W$ is given by \eqref{W} with $\Omega$ satisfying {\bf (Hx)} and $\tilde{\Omega}\ge c > 0$. Then there exists a unique pair $(a,b)\in(0,\infty)^2$ such that a nondegenerate ellipsoid $\rho_{a,b}$ is the unique energy minimizer for $W$ (up to translation). Here the assumption $\tilde{\Omega}\ge c > 0$ is automatically satisfied if $2\le s < 3$.

If $\tilde{\Omega}\ge 0$, then there exists a unique pair $(a,b)\in[0,\infty)^2\backslash\{(0,0)\}$ such that a possibly degenerate ellipsoid $\rho_{a,b}$ is the unique energy minimizer for $W$ (up to translation).
\end{theorem}

To prove the theorem, the key is to show that the potential generated by any $\rho_{a,b}$ is quadratic in its support.
\begin{lemma}\label{lem_ab}
Assume $0<s<3$ and $\Omega$ satisfies {\bf (Hx)} and $a,b>0$. Then
\begin{equation}\label{lem_ab_1}
    (|\bx|^{-s}\Omega(\bbx)+A x_1^2 + A x_2^2 + B x_3^2)*\rho_{a,b} = C,\quad \bx\in\supp\rho_{a,b}
\end{equation}
where
\begin{equation}\label{lem_ab_2}\begin{split}
    & A(a,b) = \pi\tau_s\Big(\frac{R_1}{R_3}\Big)^{2+s}\int_0^\pi \sin^3\varphi (a^2\sin^2\varphi+b^2\cos^2\varphi)^{-(2+s)/2}\tilde{\Omega}(\varphi)\rd{\varphi} \\
    & B(a,b) = 2\pi\tau_s\Big(\frac{R_1}{R_3}\Big)^{2+s}\int_0^\pi \cos^2\varphi\sin\varphi (a^2\sin^2\varphi+b^2\cos^2\varphi)^{-(2+s)/2}\tilde{\Omega}(\varphi)\rd{\varphi} \\
\end{split}\end{equation}
Furthermore, if $\tilde{\Omega}\ge 0$, then $(|\bx|^{-s}\Omega(\bbx)+A x_1^2 + A x_2^2 + B x_3^2)*\rho_{a,b}$ achieves minimum on $\supp\rho_{a,b}$.

If $\tilde{\Omega}\ge 0$, $a=0$, $b>0$, $0<s<1$, then the same is true provided that $A$ is finite. If $\tilde{\Omega}\ge 0$, $a>0$, $b=0$, $0<s<2$, then the same is true provided that $B$ is finite. 
\end{lemma}

\begin{remark}
If we do not assume the axisymmetry, then we expect a result similar to Lemma \ref{lem_ab} with the quadratic form $A x_1^2 + A x_2^2 + B x_3^2$ replaced by a general quadratic form in $\bx$, and a result similar to Theorem \ref{thm_ell} with $\rho_{a,b}$ replaced by the push-forward of $\rho_3$ by a general linear transformation.  In the 2D case, such results were obtained in \cite[Section 3]{CS22}.
\end{remark}

\begin{remark}\label{rem_ab}
As observed in \cite{CS22}, although $R_1$ is only defined for $0<s<1$, 
$$
\tau_s R_1^{2+s}=(2\pi)^{-s}\Gamma(s)\Big(\frac{2}{s(s+1)\pi} \beta\Big(\frac{1}{2},\frac{3+s}{2}\Big)\Big)^{-1}
$$
is defined for any complex  number $s$ with $\Re(s)>0$, holomorphic in $s$, and positive for real number $s>0$. This quantity appeared in \eqref{lem_ab_2} is understood in this way.
\end{remark}

To prove this lemma, we first treat the case $0<s<1$, in which we have the decomposition \eqref{decomp} and we may apply the strategy of 1D projections similar to \cite{CS22}. Then we extend it to the full range of $s$ by analytic continuation.

\subsection{The case $0<s<1$}
\label{subsec0s1}

We first give a lemma on the projection of $\rho_{a,b}$ onto one-dimensional subspaces. 

\begin{lemma}\label{lem_lin}
Assume $0<s<1$. Let $T$ be a linear transformation on $\mathbb{R}^3$ with 1D image, spanned by a unit vector $\bbx_1$. Let $\bbx_1,\bbx_2,\bbx_3$ be an orthonormal basis of $\mathbb{R}^3$. Then
\begin{equation}
    (T_\#\rho_{a,b})(y_1\bbx_1+y_2\bbx_2+y_3\bbx_3) = \lambda \rho_1(\lambda y_1)\delta(y_2)\delta(y_3)
\end{equation}
where 
\begin{equation}
    \lambda = \frac{R_1}{\max_{\bx\in\supp\rho_{a,b}}\bx\cdot\bbx_1}
\end{equation}
\end{lemma}

The proof is similar to \cite[Section 3]{CS22} and thus omitted.

\begin{proof}[Proof of Lemma \ref{lem_ab}, in the case $0<s<1$]
We first treat the case $a,b>0$. We aim to apply \eqref{decomp} to compute the generated potential. For fixed $\bxi_1\in S^2$, let $\bxi_1,\bxi_2,\bxi_3$ be an orthonormal basis. Then
\begin{equation}\label{xxi1y1}\begin{split}
    (|\bx\cdot\bxi_1|^{-s}*\rho_{a,b}) & (y_1\bxi_1+y_2\bxi_2+y_3\bxi_3) \\
    = & \int_{\mathbb{R}}|y_1-z_1|^{-s}\iint_{\mathbb{R}^2} \rho_{a,b}(z_1\bxi_1+z_2\bxi_2+z_3\bxi_3)\rd{z_2}\rd{z_3}\,\rd{z_1}.
\end{split}\end{equation}
The inner double integral $\iint_{\mathbb{R}^2} \rho_{a,b}(z_1\bxi_1+z_2\bxi_2+z_3\bxi_3)\rd{z_2}\rd{z_3}$ is the push-forward of a linear transformation onto the 1D subspace spanned by $\bxi_1$. By Lemma \ref{lem_lin}, we get
\begin{equation}
    \iint_{\mathbb{R}^2} \rho_{a,b}(z_1\bxi_1+z_2\bxi_2+z_3\bxi_3)\rd{z_2}\rd{z_3} = \frac{R_1}{r_{\bxi_1}}\rho_1\Big(\frac{R_1}{r_{\bxi_1}}z_1\Big)
\end{equation}
where
\begin{equation}
     r_{\bxi_1} := \max_{\bx\in\supp\rho_{a,b}}\bxi_1\cdot\bx  = R_3(a^2\sin^2\varphi_1+b^2\cos^2\varphi_1)^{1/2}
\end{equation}
(similar to the quantity $r_\varphi$ in \cite[Section 3]{CS22}) with $\varphi_1$ being the angle for $\bxi_1$ as in \eqref{xi}. 

The fact that $\rho_1$ minimizes the 1D interaction energy with potential $|x|^{-s}+|x|^2$ gives
\begin{equation}
    \int_{\mathbb{R}} (|y-z|^{-s}+|y-z|^2)\rho_1(z)\rd{z} = \text{constant},\quad y\in [-R_1,R_1]
\end{equation}
(and larger outside). Rescaling by $\lambda>0$, we get
\begin{equation}
    \int_{\mathbb{R}} (|y-z|^{-s}+\lambda^{2+s}|y-z|^2)\lambda\rho_1(\lambda z)\rd{z} = \text{constant},\quad y\in [-R_1/\lambda,R_1/\lambda].
\end{equation}
Applying it to \eqref{xxi1y1} with $\lambda=\frac{R_1}{r_{\bxi_1}}$, we obtain
\begin{equation}\label{xxi1}\begin{split}
    \Big(\Big(|\bx\cdot\bxi_1|^{-s}+\Big(\frac{R_1}{r_{\bxi_1}}\Big)^{2+s}|\bx\cdot\bxi_1|^2\Big)*\rho_{a,b}\Big)  (y_1\bxi_1 +y_2\bxi_2+ & y_3\bxi_3) = \text{constant},\\
    & y_1\in [-r_{\bxi_1},r_{\bxi_1}]
\end{split}\end{equation}
and it holds in $\supp\rho_{a,b}$ in particular. Integrating in $\bxi_1$ as in \eqref{decomp}, we obtain \eqref{lem_ab_1} with the quadratic parts in the potential being
\begin{equation}\begin{split}
    \tau_s\Big(\frac{R_1}{R_3}\Big)^{2+s} &  \int_{S^2}|\bx\cdot\bxi|^2 (a^2\sin^2\varphi+b^2\cos^2\varphi)^{-(2+s)/2}\tilde{\Omega}(\bxi)\rd{\bxi} \\
    = \,& \tau_s\Big(\frac{R_1}{R_3}\Big)^{2+s} \int_0^\pi \int_0^{2\pi} (x_1\sin\varphi\cos\nu + x_2\sin\varphi\sin\nu + x_3 \cos\varphi)^2\rd{\nu}\\
    & \qquad\cdot (a^2\sin^2\varphi+b^2\cos^2\varphi)^{-(2+s)/2}\tilde{\Omega}(\varphi)\sin\varphi\rd{\varphi} \\
    = \,& \pi\tau_s\Big(\frac{R_1}{R_3}\Big)^{2+s} \int_0^\pi (x_1^2\sin^2\varphi+x_2^2\sin^2\varphi+2x_3^2\cos^2\varphi) \\
    & \qquad\cdot (a^2\sin^2\varphi+b^2\cos^2\varphi)^{-(2+s)/2}\tilde{\Omega}(\varphi)\sin\varphi\rd{\varphi} \\
\end{split}\end{equation}
using the axisymmetry $\tilde{\Omega}(\bxi)= \tilde{\Omega}(\varphi)$. This gives the expressions of $A$ and $B$ as in \eqref{lem_ab_2}. We also observe that the LHS of \eqref{xxi1} is greater than the RHS constant for $y_1\notin [-r_{\bxi_1},r_{\bxi_1}]$. Therefore, if $\tilde{\Omega}\ge 0$, the LHS of \eqref{lem_ab_1} achieves minimum on $\supp\rho_{a,b}$.

If $a=0,b>0$, the the previous argument applies if $\frac{\pi}{2}\notin \supp\tilde{\Omega}$ where the last $\tilde{\Omega}$ is interpreted as a function of $\varphi$. For general $\tilde{\Omega}\ge 0$ with $A<\infty$, we necessarily have $\tilde{\Omega}(\frac{\pi}{2})=0$. Then one can approximate $\tilde{\Omega}$ by an increasing sequence of smooth nonnegative axisymmetric functions $\tilde{\Omega}_n$ with $\frac{\pi}{2}\notin \supp\tilde{\Omega}_n$. The corresponding $\Omega_n$ satisfies {\bf (Hx)} since the positivity of $\Omega_n$ is given by \eqref{lem_FT_4}. The conclusion holds for each $\tilde{\Omega}_n$, and we obtain conclusion for $\tilde{\Omega}$ by the monotone convergence theorem.

If $a>0,b=0$, one can use a similar approximation with $0\notin \supp\tilde{\Omega}_n$.

\end{proof}

\subsection{Extending the range of $s$}

Next we aim to prove Lemma \ref{lem_ab} for $1\le s <3$. Our argument is based on analytic continuation for the variable $s$. 

\begin{proof}[Proof of Lemma \ref{lem_ab} for general $s$, `equality part']
We first prove the `equality' part, i.e., \eqref{lem_ab_1} with \eqref{lem_ab_2}. We will first assume $a,b>0$, and explain why the proof also works for the exceptional case $b=0$ at the end. 

For this purpose, we first apply Lemma \ref{lem_ab} with $0<s<1$ (which is already proved in the previous subsection) and $(a,b)$ replaced by $(a/R_3,b/R_3)$. Notice that 
\begin{equation}\label{D}
    D:=\supp\rho_{a/R_3,b/R_3} = \Big\{\bx: \frac{x_1^2+x_2^2}{a^2}+\frac{x_3^2}{b^2}\le 1\Big\}
\end{equation}
is independent of $s$. Thus we obtain
\begin{equation}\label{abR3}
    (|\bx|^{-s}\Omega(\bbx;s)+A x_1^2 + A x_2^2 + B x_3^2)*\rho_{a/R_3,b/R_3} = C,\quad \bx\in D
\end{equation}
for $0<s<1$, where
\begin{equation}\label{abR3_2}\begin{split}
    & A = \pi\tau_s R_1^{2+s}\int_0^\pi \sin^3\varphi (a^2\sin^2\varphi+b^2\cos^2\varphi)^{-(2+s)/2}\tilde{\Omega}(\varphi)\rd{\varphi} \\
    & B = 2\pi\tau_s R_1^{2+s}\int_0^\pi \cos^2\varphi\sin\varphi (a^2\sin^2\varphi+b^2\cos^2\varphi)^{-(2+s)/2}\tilde{\Omega}(\varphi)\rd{\varphi} \\
\end{split}\end{equation}
Here we emphasized the dependence of $\Omega$ on $s$ with $\tilde{\Omega}$ being a fixed function on $S^2$ satisfying {\bf (H0)}. In other words, $\Omega(\bbx;s)$ is determined by $\cF[|\bx|^{-s}\Omega(\bbx;s)] = |\xi|^{-3+s}\tilde{\Omega}(\bxi)$, which is well-defined by applying Lemma \ref{lem_FT} reversely.

We will fix a choice of $(\tilde{\Omega},a,b)$ and view every quantity above as a function in $s$ (this includes $\Omega(\cdot;s),A,B,\tau_s R_1^{2+s},\rho_{a/R_3,b/R_3}$ and the constant $C$ in \eqref{abR3}). We recall from Remark \ref{rem_ab} that $\tau_s R_1^{2+s}$ is an holomorphic function in $s\in \kS$ (recalling the definition of $\kS$ in \eqref{kS}), and thus $A,B$ are well-defined and holomorphic for such $s\in \kS$. Lemma \ref{lem_FT} shows that $\Omega(\bbx;s)$ is well-defined for $s\in \kS$.

We claim that for any $\bx\in\mathbb{R}^3$, the value of the generated potential at $\bx$
\begin{equation}
    f(\bx;s) := \Big((|\bx|^{-s}\Omega(\bbx;s)+A x_1^2 + A x_2^2 + B x_3^2)*\rho_{a/R_3,b/R_3}\Big)(\bx),\quad 0<s<3
\end{equation}
can be extended to an holomorphic function in $s\in \kS$. To see this, we will treat the part with the repulsive potential $|\bx|^{-s}\Omega(\bbx;s)$ (which we denote as $f_{{\rm rep}}(s)$, suppressing the $\bx$ dependence for a moment), and the quadratic part is straightforward since $A,B$ are holomorphic. In fact, for any $s\in (0,3)$,
\begin{equation}
    \rho_{a/R_3,b/R_3}(\bx) = C_3\frac{R_3^{2+s}}{a^2 b}\Big(1-\frac{x_1^2+x_2^2}{a^2}-\frac{x_3^2}{b^2}\Big)_+^{(s-1)/2}
\end{equation}
where the prefactor $C_3\frac{R_3^{2+s}}{a^2 b} = \frac{1}{2a^2b\beta(3/2,(1+s)/2)}$ is a normalization factor, being holomorphic in $s\in \kS$. Therefore, this formula for $\rho_{a/R_3,b/R_3}(\bx)$ can be extended holomorphically to any complex number $s\in \kS$, for any fixed $\bx\in D$. We write $f_{{\rm rep}}(s)$ as
\begin{equation}
    f_{{\rm rep}}(s) = \int_D |\bx-\by|^{-s}\Omega(\overline{\bx-\by};s)C_3\frac{R_3^{2+s}}{a^2 b}\Big(1-\frac{y_1^2+y_2^2}{a^2}-\frac{y_3^2}{b^2}\Big)^{(s-1)/2}\rd{\by}
\end{equation}
Then we are allowed to take $\partial_s f_{{\rm rep}}(s)$ by differentiating the integrand. In fact, differentiating $|\bx-\by|^{-s}$ or $\Big(1-\frac{y_1^2+y_2^2}{a^2}-\frac{y_3^2}{b^2}\Big)^{(s-1)/2}$ produces an extra logarithmic singularity at $\bx$ or the boundary, which does not affect the integrability for $s\in \kS$; differentiating $C_3\frac{R_3^{2+s}}{a^2 b}$ is clearly allowed; differentiating $\Omega(\overline{\bx-\by};s)$ does not affect the integrability due to the smoothness of $\partial_s\Omega(\bbx;s)$ in $\bbx$ as shown in Lemma \ref{lem_FT}. This proves the holomorphic property of $f_{{\rm rep}}$, and thus that of $f$.

We already know from \eqref{abR3} that $f(\bx_1;s)=f(\bx_2;s)$ for any $\bx_1,\bx_2\in D$ for $0<s<1$. By the holomorphic property of $f$ in $s$, the same is true for any $s\in \kS$. In particular, it is true for $s\in [1,3)$, which gives \eqref{lem_ab_1} with \eqref{lem_ab_2} for $s\in [1,3)$.

Finally we treat the case $\tilde{\Omega}\ge 0$, $a>0$, $b=0$, $B<\infty$ for some $s=s_0\in [1,2)$. In this case the finiteness of $B$ for $s_0$ implies the same property of $B$ for any complex $s$ with $\Re(s)\in (0,s_0)$. As a result, the formula \eqref{abR3} still holds for $0<s<1$ by the previous subsection. Then we apply the same procedure to extend $f(\bx;s)$. First, $A,B$ are holomorphic for $\Re(s)\in (0,s_0)$ by the dominated convergence theorem, using the assumption that $B<\infty$ for $s=s_0$. It is easy to verify that for $b=0$, $\rho_{a/R_3,b/R_3}$ is supported on the $x_1x_2$-plane, given by
\begin{equation}
    \rho_{a/R_3,0}(\bx)=\delta(x_3)C_2\frac{R_2^{2+s}}{a^2}\Big(1-\frac{x_1^2+x_2^2}{a^2}\Big)_+^{s/2}
\end{equation}
as a rescaling of $\rho_2$. Then the generated potential from the repulsive part is
\begin{equation}\begin{split}
    f_{{\rm rep}}(s) = & \int_{y_1^2+y_2^2\le a^2} |\bx-(y_1,y_2,0)|^{-s}\Omega(\overline{\bx-(y_1,y_2,0)};s)\\
    & \cdot C_2\frac{R_2^{2+s}}{a^2}\Big(1-\frac{y_1^2+y_2^2}{a^2}\Big)^{s/2}\rd{y_1}\rd{y_2}
\end{split}\end{equation}
Then we see that $f(\bx;s)$ can be extended holomorphically to $\Re(s)\in (0,s_0)$ by the same argument as before.  

This holomorphic property implies $f(\bx_1;s)=f(\bx_2;s)$ for any $\bx_1,\bx_2\in D$ and $s\in (0,s_0)$ as before. Then the same is true for $s_0$ by taking the limit $s\rightarrow s_0^-$ using the dominated convergence theorem for the formulas for $A,B$ and $f_{{\rm rep}}$.
\end{proof}

To prove the `inequality part' of Lemma \ref{lem_ab} for general $s$, we first need a lemma on the analytic continuation of 1D generated potentials.

\begin{lemma}\label{lem_g}
For $0<s<1$, define
\begin{equation}
    g(x;s) = [(R_1^{-s-2}|x|^{-s}+|x|^2)*\bar{\rho}_1](x)-[(R_1^{-s-2}|x|^{-s}+|x|^2)*\bar{\rho}_1](1)
\end{equation}
for $x\in\mathbb{R}$, where $\bar{\rho}_1=R_1\rho_1(R_1 x)$ is a rescaling of $\rho_1$, being the unique energy minimizer for the 1D potential $R_1^{-s-2}|x|^{-s}+|x|^2$. Then for every fixed $x$, $g(x;s)$ can be extended holomorphically to $s\in \kS$. $g(\cdot;s)$ and $\partial_s g(\cdot;s)$ are bounded on $(1,T]$ for any fixed $s\in\kS$ and fixed $T>1$. $g$ also satisfies
\begin{equation}
    g(x;s) > 0,\quad \forall x>1,\,1\le s < 3
\end{equation}
\end{lemma}

\begin{proof}
Notice that $\supp\bar{\rho}_1=[-1,1]$. Clearly $g(\cdot;s)$ is an even function in $x$ and vanishes on $[-1,1]$ for any $s\in (0,1)$ by the minimizing property of $\bar{\rho}_1$. Therefore it suffices to treat the case $x>1$.

For $0<s<1$, $\bar{\rho}_1(x)=C_1R_1^{2+s}(1-x^2)_+^{(1+s)/2}$ where $C_1R_1^{2+s}=\frac{1}{\beta(1/2,(3+s)/2)}$ is a normalizing factor, holomorphic in $s$ for $\Re(s)>0$ and positive for $s>0$. For $0<s<1$ and $x>1$, we rewrite $g(x;s)$ as
\begin{equation}\label{g}\begin{split}
    g(x;s) = & \int_1^x \partial_t[(R_1^{-s-2}|\cdot|^{-s}+|\cdot|^2)*\bar{\rho}_1](t)\rd{t} \\
    = & \int_1^x \partial_t\int_{-1}^1 (R_1^{-s-2}(t-y)^{-s}+(t-y)^2)\bar{\rho}_1(y)\rd{y}\rd{t} \\
    = & \int_1^x G(t;s)\rd{t} \\
\end{split}\end{equation}
where, for $t>1$
\begin{equation}\label{G}
    G(t;s):=C_1R_1^{2+s}\int_{-1}^1 (-s R_1^{-s-2}(t-y)^{-s-1}+2(t-y))(1-y^2)^{(1+s)/2}\rd{y}
\end{equation}
with 
$s R_1^{-s-2}=\frac{2\cos\frac{s\pi}{2}}{(s+1)\pi}\beta(\frac{1}{2},\frac{3+s}{2})$, that can be extended holomorphically to $s\in\kS$. Therefore, it is clear that for every fixed $t$, $G(t;s)$ can be extended holomorphically to $s\in\kS$, given by the same formula. 

Then we estimate $G(t;s)$ for $t$ close to $1$ and $s\in\kS$. It is clear that $G(t;s)$ is bounded near $t=1$ if $\Re(s)\in (0,1)$. We claim that
\begin{equation}\label{G_est}
    |G(t;s)|+|\partial_s G(t;s)| \le C(t-1)^{(-\Re(s)+1)/2}(1+|\ln(t-1)|)^2
\end{equation}
for any $T>1$, $1<t<T$ and $\Re(s)\in [1,3)$, where $C$ depends on $T$ and uniform on compact sets for $s$. To see this, we first notice that the prefactors $C_1R_1^{2+s}$ and $s R_1^{-s-2}$ in \eqref{G} have no singularity in $s$, and thus can be ignored. Also, it suffices to treat the case of real $s$. In this case, the integral $G(t;s)$ has the same type of singularity as
\begin{equation}
    \int_{-1}^1 (t-y)^{-s-1}(1-y)^{(1+s)/2}\rd{y} = \int_0^2 (\epsilon+y)^{-s-1}y^{(1+s)/2}\rd{y},
\end{equation}
where we denote $\epsilon=t-1>0$. In the nontrivial case of sufficiently small $\epsilon$, the last integral can be estimated by cutting at $y=\epsilon$: the part $0<y<\epsilon$ can be bounded by $C\int_0^\epsilon \epsilon^{-s-1}\epsilon^{(1+s)/2}\rd{y} = C\epsilon^{(-s+1)/2}$; the part $\epsilon<y<2$ can be bounded by $C\int_\epsilon^2 y^{-s-1}y^{(1+s)/2}\rd{y} \le C\epsilon^{(-s+1)/2}$ up to an extra factor $|\ln\epsilon|$ at $s=1$. The estimate of $\partial_s G$ can be done similarly because the integrand only has extra logarithmic singularities in $t-y$ and $1-y$ (with the notations in \eqref{G}).

Therefore $G(\cdot;s)$ and $\partial_s G(\cdot;s)$ are absolutely integrable in $t\in (1,T]$ for any $s\in\kS$. This allows us to define the analytic continuation of $g(x;s)$ to $s\in\kS$ by the original formula \eqref{g}, and also shows the claimed boundedness of $g$ and $\partial_s g$.

The conclusion $g(x;s)>0$ for $x>1$ and $s\in [1,3)$ follows from the fact that $G(t;s)>0$ for $t>1$ and $s\in [1,3)$, which can be seen from \eqref{G} since the integrand and the prefactor are positive.
\end{proof}

\begin{remark}
In the last paragraph of the above proof, the positivity of $G$ seems to be subtle. In fact, the coefficient $-s R_1^{-s-2}=-\frac{2\cos\frac{s\pi}{2}}{(s+1)\pi}\beta(\frac{1}{2},\frac{3+s}{2})$ is nonnegative for $1\le s \le 3$ but becomes negative for $0<s<1$ or $3<s<4$. On one hand, the fact that $\bar{\rho}_1$ being an energy minimizer for $0<s<1$ implies $g(x;s)>0$, but this cannot be obtained from the above proof; on the other hand, this proof does not provide a nice way of extending $g$ to $s\ge 3$, or indicate its positivity there in case it is extended.
\end{remark}

\begin{proof}[Proof of Lemma \ref{lem_ab} for general $s$, `inequality part']

We follow the notations in the proof of the `equality part' of Lemma \ref{lem_ab}. For $\tilde{\Omega}\ge 0$, $s\in [1,3)$ and $a,b>0$, we will prove that the generated potential $(|\bx|^{-s}\Omega(\bbx;s)+A x_1^2 + A x_2^2 + B x_3^2)*\rho_{a/R_3,b/R_3}$ achieves minimum on $D$ (here $A,B$ follow  \eqref{abR3_2}). At the end of the proof, we will also show the same conclusion for the case $a>0,b=0,s\in [1,2),B<\infty$.

We first assume $a,b>0$. For $\bx\in\mathbb{R}^3$ and $s\in\kS$, define
\begin{equation}\label{h}\begin{split}
    h(\bx;s) = & [(|\bx|^{-s}\Omega(\bbx;s)+A x_1^2 + A x_2^2 + B x_3^2)*\rho_{a/R_3,b/R_3}](\bx)\\& -[(|\bx|^{-s}\Omega(\bbx;s)+A x_1^2 + A x_2^2 + B x_3^2)*\rho_{a/R_3,b/R_3}](D)
\end{split}\end{equation}
(the last quantity denoting the constant value of the generated potential on $D$). For every fixed $\bx\in\mathbb{R}^3$, $h(\bx;s)$ is holomorphic in $s\in\kS$ by the proof of the `equality' part of Lemma \ref{lem_ab}.

If $0<s<1$, we already proved the positivity of $h(x;s)$ in subsection \ref{subsec0s1}. Moreover,  following \eqref{xxi1} (with $a,b$ replaced by $a/R_3,b/R_3$) and tracing the constants carefully, we get
\begin{equation}\begin{split}
    & \Big(\Big(|\bx\cdot\bxi_1|^{-s}+\Big(\frac{R_1}{r_{\bxi_1}}\Big)^{2+s}|\bx\cdot\bxi_1|^2\Big)*\rho_{a/R_3,b/R_3}\Big)(y_1\bxi_1+y_2\bxi_2+y_3\bxi_3) \\
    & - \Big(\Big(|\bx\cdot\bxi_1|^{-s}+\Big(\frac{R_1}{r_{\bxi_1}}\Big)^{2+s}|\bx\cdot\bxi_1|^2\Big)*\rho_{a/R_3,b/R_3}\Big)(y_2\bxi_2+y_3\bxi_3)\\
    = & R_1^{2+s}r_{\bxi_1}^{-s}g\Big(\frac{y_1}{r_{\bxi_1}};s\Big)
\end{split}\end{equation}
where $r_{\bxi_1}:=(a^2\sin^2\varphi_1+b^2\cos^2\varphi_1)^{1/2}$. Here $g(x;s)$ is defined by the generated potential in 1D as in Lemma \ref{lem_g}. Using \eqref{decomp} and integrating in $\bxi_1$ gives
\begin{equation}\label{h1}
    h(\bx;s) = \tau_s R_1^{2+s} \int_{S^2} r_{\bxi}^{-s}g\Big(\frac{\bxi\cdot\bx}{r_{\bxi}};s\Big)\tilde{\Omega}(\bxi)\rd{\bxi}.
\end{equation}
Lemma \ref{lem_g} gives the boundedness of $g(x;s)$ and $\partial_s g(x;s)$ for fixed $s\in \kS$ and $x$ in any fixed compact set. Therefore the integral on the RHS of \eqref{h1} is well-defined for any such $s$, and one can differentiate with respect to $s$ by differentiating the integrand (noticing that $r_{\bxi}$ is always bounded from below for $a,b>0$). We recall from Remark \ref{rem_ab} that $\tau_s R_1^{2+s}$ is an holomorphic function in $s\in \kS$. Therefore the RHS of \eqref{h1} is holomorphic in $s$, and must agree with $h$ as in its definition \eqref{h} for any $s\in \kS$ since the latter is also holomorphic in $s$. Since the integrand and prefactor in \eqref{h1} are nonnegative for $s\in [1,3)$ by Lemma \ref{lem_g} and Remark \ref{rem_ab} respectively, we see that $h(\bx;s)\ge 0$ as in \eqref{h}, which finishes the proof in the case $a,b>0$.

In the case $a>0,b=0,s\in [1,2),B<\infty$, we will take the limit $b\rightarrow 0^+$ of the previous case. Fix $a$ and $s$, and denote $D_b,A_b,B_b,h_b$ as the domains/quantities depending on $b$.  

Notice that $D_0\subset D_b$ for any $b>0$ due to its definition \eqref{D}. We showed in the `equality part' of Lemma \ref{lem_ab} that $(|\bx|^{-s}\Omega(\bbx;s)+A_b x_1^2 + A_b x_2^2 + B_b x_3^2)*\rho_{a/R_3,b/R_3}$ is constant in $D_b$ for any $b\ge 0$, and thus the same is true in $D_0$. We denote this constant as $C_b$, which is indeed twice the total energy of $\rho_{a/R_3,b/R_3}$ for the potential $|\bx|^{-s}\Omega(\bbx;s)+A_b x_1^2 + A_b x_2^2 + B_b x_3^2$. We also have $A_b\rightarrow A_0$ and $B_b\rightarrow B_0$ as $b\rightarrow 0^+$ by the dominated convergence theorem, since $A_0,B_0<\infty$. Therefore, by the weak lower-semicontinuity of the energy functional, we see that $C_0 \le \liminf_{b\rightarrow0^+} C_b$ since $\rho_{a/R_3,b/R_3}$ converges weakly to $\rho_{a/R_3,0}$.

On the other hand, for $\bx_1\notin D_0$, we have $\bx_1\notin D_b$ for sufficiently small $b>0$. Therefore $[(|\bx|^{-s}\Omega(\bbx;s)+A_b x_1^2 + A_b x_2^2 + B_b x_3^2)*\rho_{a/R_3,b/R_3}](\bx_1)$, as a function of $b$, is continuous at $b=0$. 

The above two limits allow us to pass to the limit on the previously obtained result $h_b(\bx_1;s)\ge 0$ and obtain the desired result $h_0(\bx_1;s)\ge 0$.
\end{proof}

\subsection{Finalizing the proof of Theorem \ref{thm_ell}}

\begin{lemma}\label{lem_ab2}
Assume $0<s<2$ and $\tilde{\Omega}$ a nonnegative axisymmetric smooth function on $S^2$. Let $A(a,b),B(a,b)$ be given by \eqref{lem_ab_2}. Then one of the following is true:
\begin{itemize}
    \item[1.] There exists $(a,b)\in(0,\infty)^2$ such that $A(a,b)=B(a,b)=1$.
    \item[2.] There exists $b\in(0,\infty)$ such that $A(0,b)\le 1$, $B(0,b)=1$.
    \item[3.] There exists $a\in(0,\infty)$ such that $A(a,0)= 1$, $B(a,0)\le 1$.
\end{itemize}
If we instead assume $0<s<3$ and $\tilde{\Omega}\ge c> 0$, then item 1 holds.
\end{lemma}

\begin{proof}
We first treat the case $0<s<3$ and $\tilde{\Omega}\ge c> 0$. By homogeneity, it suffices to find the solution to
\begin{align}\label{fb}
    &f(b) := \frac{A(1,b)}{B(1,b)} =     \frac{ \int_0^\pi \sin^2\varphi\, w(\varphi)\rd{\varphi}} {2\int_0^\pi \cos^2\varphi\, w(\varphi)\rd{\varphi}} = 1\\
    &\quad \mbox{with } w(\varphi):= \sin \varphi (\sin^2\varphi+b^2\cos^2\varphi)^{-(2+s)/2}\tilde{\Omega}(\varphi).
\end{align}
When $b\rightarrow 0^+$, the numerator converges to $\int_0^\pi |\sin\varphi|^{1-s}\tilde{\Omega}(\varphi)\rd{\varphi}$ and thus remains bounded (for $0<s<2$), while the denominator goes to infinity because it has a singularity like $|\sin\varphi|^{-1-s}$ and $\tilde{\Omega}\ge c > 0$. Therefore $\lim_{b\rightarrow 0^+}f(b) = 0$. For the case $2\le s < 3$, both the numerator and the denominator go to infinity, but it is clear that the denominator is much larger near the singularities $\varphi=0,\pi$, and thus $\lim_{b\rightarrow 0^+}f(b) = 0$ is still true. Similarly $\lim_{b\rightarrow \infty}f(b) = \infty$. Since $f$ is continuous, we see that there exists $b>0$ with $f(b)=1$.

Then we treat the general case $\tilde{\Omega}\ge 0$ with $0<s<2$ assumed. As $b\rightarrow 0^+$, the numerator of \eqref{fb} still converges to $\int_0^\pi |\sin\varphi|^{1-s}\tilde{\Omega}(\varphi)\rd{\varphi}$, while the denominator converges to $2\int_0^\pi \cos^2\varphi(\sin\varphi)^{-1-s} \tilde{\Omega}(\varphi)\rd{\varphi}$, the latter being a positive number or infinity. Therefore $\lim_{b\rightarrow 0^+}f(b)\in [0,\infty)$. Similarly $\lim_{b\rightarrow \infty}f(b) \in (0,\infty]$. 

If $\lim_{b\rightarrow 0^+}f(b) \ge 1$, then $A(1,0)\ge B(1,0)$, and thus one can find $a>0$ such that item 3 in the statement of the lemma holds by using the homogeneity. If $\lim_{b\rightarrow \infty}f(b) \le 1$, then one can find $b>0$ such that item 2 holds. If $\lim_{b\rightarrow 0^+}f(b) < 1$ and $\lim_{b\rightarrow \infty}f(b) > 1$, then item 1 holds as in the previously considered case $\tilde{\Omega}\ge c> 0$.

\end{proof}

\begin{proof}[Proof of Theorem \ref{thm_ell}]
If $\tilde{\Omega}\ge c > 0$, then Lemma \ref{lem_ab2} gives a pair $(a,b)\in (0,\infty)^2$ such that $A(a,b)=B(a,b)=1$. Then Lemma \ref{lem_ab} and Lemma \ref{lem_EL} show that $\rho_{a,b}$ is the unique energy minimizer for $W$ in \eqref{W}, since $W$ has the LIC property by Theorem \ref{thm_LICequiv}.

Then we treat the general case of $\tilde{\Omega}\ge 0$. We may assume $0<s<2$ because $\tilde{\Omega}$ is always strictly positive in the case $2\le s<3$ due to \eqref{lem_FT_2}, \eqref{lem_FT_2b} and the assumed strict positivity of $\Omega$.  One of the three items of Lemma \ref{lem_ab2} must happen, and the conclusion can be obtained as before in the case of item 1. If item 2 happens, i.e., there exists $b\in(0,\infty)$ such that $A(0,b)\le 1$, $B(0,b)=1$, then Lemma \ref{lem_ab} shows that $(|\bx|^{-s}\Omega(\bbx)+A(0,b) x_1^2 + A(0,b) x_2^2 + x_3^2)*\rho_{0,b}$ achieves minimum on $\supp\rho_{0,b}$. Since $A(0,b)\le 1$, the same is true for $W*\rho_{0,b} = (|\bx|^{-s}\Omega(\bbx)+x_1^2 + x_2^2 + x_3^2)*\rho_{0,b}$. Therefore Lemma \ref{lem_EL} show that $\rho_{a,b}$ is the unique energy minimizer for $W$. The case of item 3 of Lemma \ref{lem_ab2} is similar.

\end{proof}

\begin{remark}\label{rem_abuni}
As a byproduct of the uniqueness part of Theorem \ref{thm_ell}, the pair $(a,b)$ as in Lemma \ref{lem_ab2} is unique.
\end{remark}

\section{Collapse to lower dimensions for large $\alpha$}\label{sec_lowerdim}

In this section we prove that the energy minimizers for the potential $W_\alpha$ (defined in \eqref{Walpha}) will collapse into 1D/2D for sufficiently large $\alpha$, for certain ranges of $s$ and under certain non-degeneracy conditions. The candidates of the lower-dimensional minimizers $\rho_{\textnormal{1D}}$ and $\rho_{\textnormal{2D}}$, as given by \eqref{rho1d2d}, are the minimizers of restricted isotropic Riesz interaction energies.

\begin{theorem}\label{thm_collapse1}
For fixed $0<s<1$, there exists $C_*$ such that the following holds: For $\Omega$ satisfying {\bf (Hx)}, if the minimum of $\Omega$ is achieved at $\theta=0$ with $\Omega(0)=1$ and the non-degeneracy condition
\begin{equation}\label{thm_collapse1_1}
    \Omega(\theta) \ge 1 + C_*|\theta|^2,\quad \forall\theta\in \Big[0,\frac{\pi}{2}\Big]
\end{equation}
then $\rho_{\textnormal{1D}}$ is the unique energy minimizer.

As a result, if $\omega$ satisfies {\bf (hx)}, achieves minimum at $\omega(0)=0$, and satisfies the non-degeneracy condition \begin{equation}
    \omega(\theta) \ge  c_\omega|\theta|^2,\quad \forall\theta\in \Big[0,\frac{\pi}{2}\Big]
\end{equation}
for some $c_\omega>0$. Then there exists a unique $0<\alpha_*\le C_*/c_\omega$, such that for any $\alpha>\alpha_*$, $\rho_{\textnormal{1D}}$ is the unique energy minimizer for $W_\alpha$, and for any $0\le \alpha < \alpha_*$, $\rho_{\textnormal{1D}}$ is not an energy minimizer for $W_\alpha$.
\end{theorem}

\begin{theorem}\label{thm_collapse2}
For fixed $0<s<2$, there exists $C_*$ such that the following holds: For $\Omega$ satisfying {\bf (Hx)}, if the minimum of $\Omega$ is achieved at $\theta=\frac{\pi}{2}$ with $\Omega(\frac{\pi}{2})=1$ and the non-degeneracy condition
\begin{equation}
    \Omega(\theta) \ge 1 + C_*\Big|\theta-\frac{\pi}{2}\Big|^2,\quad \forall\theta\in \Big[0,\frac{\pi}{2}\Big]
\end{equation}
then $\rho_{\textnormal{2D}}$ is the unique energy minimizer.

As a result, if $\omega$ satisfies {\bf (hx)}, achieves minimum at $\omega(\frac{\pi}{2})=0$, and satisfies the non-degeneracy condition \begin{equation}
    \omega(\theta) \ge  c_\omega\Big|\theta-\frac{\pi}{2}\Big|^2,\quad \forall\theta\in \Big[0,\frac{\pi}{2}\Big]
\end{equation}
for some $c_\omega>0$. Then there exists a unique $0<\alpha_*\le C_*/c_\omega$, such that for any $\alpha>\alpha_*$, $\rho_{\textnormal{2D}}$ is the unique energy minimizer for $W_\alpha$, and for any $0\le \alpha < \alpha_*$, $\rho_{\textnormal{2D}}$ is not an energy minimizer for $W_\alpha$.
\end{theorem}

\begin{remark}
In the above two theorems, the requirements on the range of $s$ are necessary, because a larger $s$ would make $\rho_{\textnormal{1D}}$ (respectively, $\rho_{\textnormal{2D}}$) having infinite energy, and cannot be an energy minimizer for any $\alpha$. In particular, they are not applicable to the case $1\le s < 2$ and $\omega$ achieves minimum at $\omega(0)=0$, which will be studied in detail in the next section.
\end{remark}

\begin{remark}
The `non-degeneracy condition' in Theorem \ref{thm_collapse1} is necessary. In fact, if instead one has $\Omega(\theta)\le 1+C|\theta|^\kappa$ for some $\kappa>2$, then one can apply \cite[Theorem 5.5]{CS22} in the $x_1x_3$-plane to show that $\rho_{\textnormal{1D}}$ is not a Wasserstein-$\infty$ local minimizer. We expect similar result also holds for Theorem \ref{thm_collapse2}.
\end{remark}

The proof is based on a comparison argument similar to \cite[Theorem 5.1]{CS22}, once we have the following two lemmas. In fact, to prove Theorem \ref{thm_collapse1}, we observe that \eqref{thm_collapse1_1} implies $\Omega(\theta)\ge \Omega_{*,1}(\theta)$ (the latter given by Lemma \ref{lem_omegast1}) for any $\theta$ if $C_*$ is sufficiently large. Also, for the corresponding energies, $E[\rho_{\textnormal{1D}}]=E_{*,1}[\rho_{\textnormal{1D}}]$ since $\Omega(0)= \Omega_{*,1}(0)$. Lemma \ref{lem_omegast1} implies that $\rho_{\textnormal{1D}}$ is the unique minimizer for $\Omega_{*,1}$, and then a comparison argument shows that it is also the unique minimizer for $\Omega$. Theorem \ref{thm_collapse2} can be proved in the same way by using $\Omega_{*,2}$ from Lemma \ref{lem_omegast2}.

\begin{lemma}\label{lem_omegast1}
Assume $0<s<1$. There exists $\Omega_{*,1}$ satisfying {\bf (Hx)} with
\begin{itemize}
    \item $\Omega_{*,1}$ achieves its minimal value 1 only at $\theta=0$.
    \item $\tilde{\Omega}_{*,1}\ge 0$.
    \item The associated potential satisfies that $(|\bx|^{-s}\Omega_{*,1}(\bbx)+|\bx|^2)*\rho_{\textnormal{1D}}$ achieves minimum on $\supp\rho_{\textnormal{1D}}$.
\end{itemize}
\end{lemma}

\begin{proof}

We construct $\Omega_{*,1}$ via constructing $\tilde{\Omega}_{*,1}$ and applying the formula \eqref{lem_FT_4}. We will construct $\tilde{\Omega}_{*,1}$ as a smooth nonnegative axisymmetric even function on $S^2$, supported near $\varphi=0$ (thus also $\varphi=\pi$). In this case, one is allowed to apply Lemma \ref{lem_ab} with $a=0$ and $b=R_1/R_3$ (for which $\rho_{a,b}=\rho_{\textnormal{1D}}$) since $\tilde{\Omega}_{*,1}$ vanishes near $\pi/2$ and $A$ is clearly finite. The condition $\Omega_{*,1}(0)=1$ is equivalent to $\rho_{\textnormal{1D}}$ being a steady state for $|\bx|^{-s}\Omega_{*,1}(\bbx)+|\bx|^2$, i.e., $B=1$. This gives the requirement
\begin{equation}\label{omegast1}
    1= 2\pi\tau_{s}\int_0^\pi|\cos\varphi|^{-s}\sin\varphi\,\tilde{\Omega}_{*,1}(\varphi) \rd{\varphi}\,.
\end{equation}
Then $(|\bx|^{-s}\Omega_{*,1}(\bbx)+|\bx|^2)*\rho_{\textnormal{1D}}$ achieves minimum on $\supp\rho_{\textnormal{1D}}$ as long as
\begin{equation}
    A = \pi\tau_s\int_0^\pi\tan^2\varphi |\cos\varphi|^{-s}\sin\varphi\, \tilde{\Omega}_{*,1}(\varphi)\rd{\varphi} < 1
\end{equation}
comparing again to the conclusion in Lemma \ref{lem_ab}.

Let us now show that indeed we can find a function $\tilde{\Omega}_{*,1}$ such that $A<1$.
Compared with the last integral in \eqref{omegast1}, there is an extra factor $\tan^2\varphi$ in the definition of $A$, which is close to zero near $\varphi=0,\pi$. Therefore, by taking $|\cos\varphi|^{-s}\sin\varphi\, \tilde{\Omega}_{*,1}(\varphi)$ as a mollifier supported in $\varphi\in [\epsilon/2,\epsilon]\cup[\pi-\epsilon,\pi-\epsilon/2]$ with $\epsilon>0$ small and satisfying \eqref{omegast1}, we can guarantee that $A<1$ is also satisfied. Then we may divide by $|\cos\varphi|^{-s}\sin\varphi$ and obtain the desired $\tilde{\Omega}_{*,1}$. 

In fact, the only property not checked yet is that $\Omega_{*,1}(\theta)\ge 1$, equality only achieved at $0$ and $\pi$. To see this, one notices that the LIC property and the fact that $(|\bx|^{-s}\Omega_{*,1}(\bbx)+|\bx|^2)*\rho_{\textnormal{1D}}$ achieves minimum on $\supp\rho_{\textnormal{1D}}$ implies that $\rho_{\textnormal{1D}}$ is the unique energy minimizer, by Lemma \ref{lem_EL}. If $\Omega_{*,1}(\theta)\le 1$ at some $\theta$ other than $0,\pi$, then a rotated version of $\rho_{\textnormal{1D}}$ along the direction of $\theta$ would have equal or smaller energy than $\rho_{\textnormal{1D}}$, a contradiction.
\end{proof}

\begin{lemma}\label{lem_omegast2}
Assume $0<s<2$. There exists $\Omega_{*,2}$ satisfying {\bf (Hx)} with
\begin{itemize}
    \item $\Omega_{*,2}$ achieves its minimal value 1 only at $\theta=\pi/2$.
    \item $\tilde{\Omega}_{*,2}\ge 0$.
    \item The associated potential satisfies that $(|\bx|^{-s}\Omega_{*,2}(\bbx)+|\bx|^2)*\rho_{\textnormal{2D}}$ achieves minimum on $\supp\rho_{\textnormal{2D}}$.
\end{itemize}

\end{lemma}

\begin{proof}

We construct $\Omega_{*,2}$ by
\begin{equation}\label{tomegast2}
    \tilde{\Omega}_{*,2}(\bxi) = \int_{S^2}\delta(\bxi\cdot\bby)\psi(\bby)\rd{\bby}
\end{equation}
where $\psi$ is a smooth nonnegative axisymmetric even function on $S^2$, supported near $\varphi=0$. Such $\tilde{\Omega}_{*,2}$ is supported near $\varphi=\pi/2$. Then, for any $0<s<2$, one can apply Lemma \ref{lem_FTpsi} reversely and recover $\Omega_{*,2}$ as
\begin{equation}\label{omegast2_1}
    \Omega_{*,2}(\bbx) = c_{2-s,\textnormal{2D}}\int_{S^2}(1-|\bby\cdot\bbx|^2)^{-s/2}\psi(\bby)\rd{\bby}\,.
\end{equation}

We may apply Lemma \ref{lem_ab} with $a=R_2/R_3$ and $b=0$ (for which $\rho_{a,b}=\rho_{\textnormal{2D}}$) since $\tilde{\Omega}_{*,2}$ vanishes near $\varphi=0$ and $B$ is clearly finite. The condition $\Omega_{*,2}(\pi/2)=1$ is equivalent to $\rho_{\textnormal{2D}}$ being a steady state for the potential $|\bx|^{-s}\Omega_{*,2}(\bbx)+|\bx|^2$, i.e., $A=1$. This gives the requirement
\begin{equation}\label{omegast2}
    1=  \pi\tau_{s}(R_1/R_2)^{2+s}\int_0^\pi\sin^{-s+1}\varphi\,\tilde{\Omega}_{*,2}(\varphi) \rd{\varphi} \,.
\end{equation}
Then $(|\bx|^{-s}\Omega_{*,2}(\bbx)+|\bx|^2)*\rho_{\textnormal{2D}}$ achieves minimum on $\supp\rho_{\textnormal{2D}}$ as long as
\begin{equation}
    B = 2\pi\tau_s(R_1/R_2)^{2+s}\int_0^\pi\cot^2\varphi \sin^{-s+1}\varphi\,\tilde{\Omega}_{*,2}(\varphi) \rd{\varphi} < 1\,,
\end{equation}
similarly to the previous lemma.
Compared with the last integral in \eqref{omegast2}, there is an extra factor $\cot^2\varphi$ in the definition of $B$, which is close to zero near $\varphi=\pi/2$. Therefore, by taking $\sin^{-s+1}\varphi\,\tilde{\Omega}_{*,2}(\varphi)$ as a mollifier with sufficiently small support near $\varphi=\pi/2$ satisfying \eqref{omegast1} (which is possible in the form \eqref{tomegast2}), we can guarantee $B<1$. 

By choosing $\tilde{\Omega}_{*,2}$ more carefully, one can guarantee that $\Omega_{*,2}$ achieves its minimal value only at $\theta=\pi/2$. In fact, the integral kernel 
$$
(1-|\bbx\cdot(0,0,1)|^2)^{-s/2}=(\sin\theta)^{-s/2}
$$ 
in \eqref{omegast2_1} achieves its unique minimum at $\theta=\pi/2$, with the strict convexity $(\partial_{\theta\theta}(\sin\theta)^{-s/2})(\pi/2)>0$. If $\psi$ is sufficiently concentrated near $\varphi=0$, the function $\Omega_{*,2}$ given by \eqref{omegast2_1}, as a slightly mollified version of $(1-|\bbx\cdot(0,0,1)|^2)^{-s/2}$, also have positive second $\theta$-derivative near $\theta=\pi/2$. Thus it achieves local minimum at $\theta=\pi/2$ in an interval $\theta\in [\pi/2-\epsilon,\pi/2+\epsilon]$ with  $\epsilon$ small. Furthermore, since the integral kernel $(\sin\theta)^{-s/2}$ has strictly larger values on the complement of $[\pi/2-\epsilon,\pi/2+\epsilon]$ compared to its value at $\theta=\pi/2$, the same is true for its mollification. Therefore $\Omega_{*,2}$ achieves its minimal value only at $\theta=\pi/2$.
\end{proof}

\section{Expansion of energy minimizers for large $\alpha$}\label{sec_expand}

In this section we consider the cases where the potential is too singular to allow the energy minimizer to collapse to lower dimension distributions. In particular, we are interested in the energy minimizers for $W_\alpha$ in \eqref{Walpha} in the case $1\le s < 2$, $\omega$ satisfying {\bf (hx)} and achieving minimum at $\theta=0$. In this case, the minimizers are not allowed to concentrate on the preferred $x_3$-direction. As the parameter $\alpha$ in \eqref{Walpha} gets large, we will analyze the expansion of the energy minimizer as $\alpha$ increases.

\subsection{Existence of LIC/non-LIC potentials}

We first explore the possibility of the signs of $\tilde{\omega}$ for $1\le s<2$ and $\omega$ satisfying {\bf (hx)}, achieving minimum at $\theta=0$.

\begin{theorem}\label{thm_s12exist}
Assume $1< s < 2$, and fix $\theta_0\in (0,\pi/2)$. 
\begin{itemize}
    \item There exists $\omega_1$ satisfying {\bf (hx)}, achieving minimal value $\omega_1(\theta)=0$ on $\theta\in [0,\pi/2-\theta_0]$, and the angle function for its Fourier transform $\tilde{\omega}_1$ is strictly positive. \item There exists $\omega_2$ satisfying {\bf (hx)}, achieving minimal value $\omega_2(\theta)=0$ on $\theta\in [0,\pi/2-\theta_0]$, and $\tilde{\omega}_2$ is sign-changing.
\end{itemize}
If $s=1$, then the same are true with `strictly positive' in item 1 replaced by `nonnegative' (item 1 would be false without this replacement).
\end{theorem}

\begin{remark}
As a complement of this theorem, we notice that it is not possible to have $\omega$ satisfying {\bf (hx)}, not identically zero, with $\tilde{\omega}\le 0$. In fact, for such $\omega$, we have the Fourier formula for the interaction energy 
\begin{equation}
    \int_{\mathbb{R}^3}\int_{\mathbb{R}^3} |\bx-\by|^{-s}\omega(\overline{\bx-\by})\rho(\by)\rd{\by}\rho(\bx)\rd{\bx}=\int_{\mathbb{R}^3} |\xi|^{-3+s}\tilde{\omega}(\bxi)|\hat{\rho}(\xi)|^2\rd{\xi}
\end{equation}
at least for compactly supported smooth $\rho$. If $\omega\ge 0$ and not identically zero, then the LHS is positive provided that $\rho$ is nonnegative and compactly supported. This shows $\tilde{\omega}\le 0$ cannot hold because otherwise the RHS would be nonpositive.
\end{remark}

\begin{proof}

The case $s=1$ can be easily treated by using \eqref{lem_FT_4b}. In fact, for $\omega_1$, we take $\tilde{\omega}_1$ as a nonnegative smooth function with $\tilde{\omega}_1(\bxi)=\tilde{\omega}_1(-\bxi)$ and supported on $\varphi\in [0,\theta_0]\cup[\pi-\theta_0,\pi]$. For $\omega_2$, we define $\tilde{\omega}_2$ by taking $\tilde{\omega}_1$ and making it slightly negative for $\varphi\in [0,\theta_0/2]\cup [\pi-\theta_0/2,\pi]$. From \eqref{lem_FT_4b}, it is clear that this modification guarantees that $\omega_2$ is nonnegative.

In the rest of the proof, we assume $1<s<2$.

One can take $\omega_1$ as the $\Omega^\epsilon$ in Lemma \ref{lem_FTpsi2} with $\epsilon\le \theta_0$, and all the desired properties are consequences of Lemma \ref{lem_FTpsi2} and the explicit expression \eqref{tomega_psi}.

To get the desired $\omega_2$ with $\tilde{\omega}_2$ sign-changing, we define
\begin{equation}
    \omega_2(\bbx)=\Omega^{\epsilon_1}-\kappa\Omega^{\epsilon_2} = \int_{S^2}\delta(\bbx\cdot\bby)(\psi_{\epsilon_1}(\bby)-\kappa\psi_{\epsilon_2}(\bby))\rd{\bby}\,,
\end{equation}
where $\psi_\epsilon$ is as defined in Lemma \ref{lem_FTpsi2}, and the parameters $0<2\epsilon_2\le \epsilon_1\le \theta_0$ and $\kappa>0$ to be chosen. Item 1 of Lemma \ref{lem_FTpsi2} shows that $\omega_2(\theta)=0$ on $\theta\in [0,\pi/2-\theta_0]$. Item 2 of Lemma \ref{lem_FTpsi2}, together with the condition $2\epsilon_2\le \epsilon_1$, shows that $\omega_2\ge 0$ as long as
\begin{equation}\label{kappacond1}
    \epsilon_1^{-1} \ge C\kappa \epsilon_2^{-1}\,.
\end{equation}
Item 3 of Lemma \ref{lem_FTpsi2} shows that $\tilde{\omega}_2(0)<0$ as long as
\begin{equation}\label{kappacond2}
    \epsilon_1^{-1} < c\kappa^{1/(3-s)} \epsilon_2^{-1}\,.
\end{equation}
Conditions \eqref{kappacond1} and \eqref{kappacond2} can be simultaneously satisfied as long as $c\kappa^{1/(3-s)} > C\kappa$ (where $C$ and $c$ are as in these conditions). This is clearly possible by choosing $\kappa$ sufficiently small, since $1<s<2$. We may choose it such that $C\kappa \le 1/2$, and then choosing $\epsilon_1=\theta_0$, $\epsilon_2 = \theta_0 C\kappa$ will satisfy \eqref{kappacond1}, \eqref{kappacond2} and $0<2\epsilon_2\le \epsilon_1\le \theta_0$.

\end{proof}

For the purpose of a later application, we also give an existence result which is a variant of $\omega_1$ in Theorem \ref{thm_s12exist}.
\begin{lemma}\label{lem_s12exist2}
Assume $1\le s < 2$, and fix $\varphi_0\in (0,\pi/2)$. There exists $\omega_3$ satisfying {\bf (hx)}, achieving minimal value only at $\omega_3(0)=0$, and $\tilde{\omega}_3$ is nonnegative and supported inside $\varphi\in [\pi/2-\varphi_0,\pi/2+\varphi_0]$ with $\tilde{\omega}_3(\pi/2)>0$.
\end{lemma}

\begin{proof}
For $s=1$, we have the formula \eqref{lem_FT_4b}. The desired conditions for $\omega_3$ are clearly satisfied if we take $\tilde{\omega}_3$ to be a nonnegative smooth axisymmetric function with $\tilde{\omega}_3(\bxi)=\tilde{\omega}_3(-\bxi)$, supported on $\varphi\in [\pi/2-\varphi_0,\pi/2+\varphi_0]$, with $\tilde\omega_3(\pi/2)=0$.

In the rest of this proof, we assume $1<s<2$. We define
\begin{equation}
    \tilde{\omega}_3(\bxi) = \int_{S^2}\delta(\bxi\cdot\bby)(\psi_{\epsilon_1}(\bby)-\kappa\psi_{\epsilon_2}(\bby))\rd{\bby}
\end{equation}
where $\psi_\epsilon$ is as defined in Lemma \ref{lem_FTpsi2} so that items 4 and 5 of the lemma are satisfied, and the parameters $0<2\epsilon_2\le \epsilon_1\le \varphi_0$ and $\kappa>0$ to be chosen. Notice that this strategy is similar to the one in Theorem \ref{thm_s12exist} but in Fourier space.
Then we apply Lemma \ref{lem_FTpsi} reversely and obtain
\begin{equation}
    \omega_3(\bbx) = c_{2-s,2D}\int_{S^2}(1-|\bbx\cdot\bby|^2)^{-s/2}(\psi_{\epsilon_1}(\bby)-\kappa\psi_{\epsilon_2}(\bby))\rd{\bby}
\end{equation}
Item 1 of Lemma \ref{lem_FTpsi2} shows that $\tilde{\omega}_3(\varphi)=0$ on $\varphi\in [0,\pi/2-\varphi_0]$. Item 2 of Lemma \ref{lem_FTpsi2}, together with the condition $2\epsilon_2\le \epsilon_1$, shows that $\tilde{\omega}_3\ge 0$ with $\tilde{\omega}_3(\pi/2)>0$ as long as
\begin{equation}\label{kappacond1_new}
    \epsilon_1^{-1} \ge C\kappa \epsilon_2^{-1}\,.
\end{equation}
Item 3 of Lemma \ref{lem_FTpsi2} shows that $\omega_3(0)=0$ as long as
\begin{equation}\label{kappacond2_new}
    \epsilon_1^{-s} = c(\epsilon_1,\epsilon_2)\kappa \epsilon_2^{-s}
\end{equation}
where $c(\epsilon_1,\epsilon_2)$ is a positive constant depending on $\epsilon_1$ and $\epsilon_2$, being uniformly bounded from above and below. Choosing $\kappa$ according to \eqref{kappacond2_new}, we see that \eqref{kappacond1_new} reduces to
\begin{equation}
    \epsilon_1^{s-1} \ge \frac{C}{c(\epsilon_1,\epsilon_2)} \epsilon_2^{s-1}
\end{equation}
We fix the choice $\epsilon_1 = \varphi_0$, and then this condition, together with $2\epsilon_2\le \epsilon_1$, are automatically satisfied as long as $\epsilon_2$ is sufficiently small.

Then we claim that $\omega_3$ achieves minimal value only at $\omega_3(0)=0$ if  $\epsilon_2$ is sufficiently small. To see this, we denote the two parts of $\omega_3$ as $\omega_3=\omega_{3,1}-\omega_{3,2}$ with
$$
\omega_{3,1}(\bbx) = c_{2-s,2D}\int_{S^2}(1-|\bbx\cdot\bby|^2)^{-s/2}\psi_{\epsilon_1}(\bby)\rd{\bby}
$$
and
$$
    \omega_{3,2}(\bbx) = c_{2-s,2D}\int_{S^2}(1-|\bbx\cdot\bby|^2)^{-s/2}\kappa\psi_{\epsilon_2}(\bby)\rd{\bby}.
$$
Here $\omega_{3,1}$ is a fixed smooth positive axisymmetric function, and thus
\begin{equation}
    \min \omega_{3,1} >0,\quad \omega_{3,1}(\theta)\ge \omega_{3,1}(0)(1-C_{3,1}\theta^2),\quad\forall\theta\in [0,\pi/2]
\end{equation}
for some $C_{3,1}>0$. 

We now make use of Lemma \ref{lem_FTpsi2} to compare $\omega_{3,1}$ to $\omega_{3,2}$ by varying $\epsilon_2$. First, by requiring $\epsilon_2$ small so that 
\begin{equation}
    c_{\psi,1}/\epsilon_2^2 > C_{3,1},
\end{equation}
we may apply item 5 of Lemma \ref{lem_FTpsi2} to $\omega_{3,2}$ and conclude that $\omega_{3,1}(\theta)>\omega_{3,2}(\theta)$ for any $0<\theta\le c_{\psi,2}\epsilon_2$ due to $\omega_{3,1}(0)=\omega_{3,2}(0)$.

We then choose $\theta_1>0$ small enough (independent of $\epsilon_2$) so that 
\begin{equation}
    C_{3,1}\theta_1^2 < c_{\psi,1}c_{\psi,2}^2\,.
\end{equation}
This guarantees that $\omega_{3,1}(\theta)>\omega_{3,2}(\theta)$ for any $c_{\psi,2}\epsilon_2<\theta\le \theta_1$ because
\begin{align*}
    \omega_{3,1} (\theta) \ge & \, \omega_{3,1}(0)(1-C_{3,1}\theta^2) \ge \omega_{3,1}(0)(1-C_{3,1}\theta_1^2) = \omega_{3,2}(0)(1-C_{3,1}\theta_1^2) \\
    > & \,\omega_{3,2}(0)(1-c_{\psi,1}c_{\psi,2}^2) = \omega_{3,2}(0)\Big(1-c_{\psi,1}\frac{(c_{\psi,2}\epsilon)^2}{\epsilon^2}\Big) \\
    \ge &\, \omega_{3,2}(c_{\psi,2}\epsilon) \ge \omega_{3,2}(\theta)\,,
\end{align*}
where items 5 and 4 of Lemma \ref{lem_FTpsi2} are used in the last two inequalities respectively.

If we further decrease $\epsilon_2$, \eqref{kappacond2_new} shows that $\kappa$ scales like $\epsilon_2^s$ as $\epsilon_2\rightarrow 0$. Therefore $\omega_{3,2}(\theta)$ converges to zero uniformly on $\theta\in [\theta_1,\pi/2]$ as $\epsilon_2\rightarrow 0$. This means if $\epsilon_2$ is sufficiently small, we have $\omega_{3,2}(\theta) < \min \omega_{3,1}$ on $\theta\in [\theta_1,\pi/2]$. Thus we conclude that $\omega_{3,2}<\omega_{3,1}$ for any $\theta\in (0,\pi/2]$, i.e., $\omega_3$ achieves its minimum only at $\omega_3(0)=0$.

\end{proof}

\begin{remark}
As a byproduct, this lemma shows that item 1 in Lemma \ref{lem_omegast2} is not a consequence of items 2 and 3 therein. In fact, the function $\omega_3$ in Lemma \ref{lem_s12exist2} achieves its minimum at $\omega_3(0)=0$ and satisfies $\omega_3(\pi/2)>0$ and $\tilde{\omega}_3\ge 0$. By taking a constant multiple, we may assume $\omega_3(\pi/2)=1$, and thus we have $A=1$ as in the proof of Lemma \ref{lem_omegast2}. Also we have $B<1$ because Lemma \ref{lem_s12exist2} guarantees that $\tilde{\omega}_3$ is concentrated near $\pi/2$. Therefore $\omega_3$ satisfies items 2 and 3 of Lemma \ref{lem_omegast2} but not item 1. This is in contrast with Lemma \ref{lem_omegast1}, in which item 1 is a consequence of items 2 and 3.
\end{remark}

\subsection{The LIC case}

Then we analyze the expansion phenomenon in case $\Omega_\alpha$ is always LIC for any $\alpha\ge 0$, i.e., the case $\tilde{\omega}\ge 0$. In item 1 of Theorem \ref{thm_s12exist}, we have seen that for $1\le s<2$ there exists such $\omega$ satisfying {\bf (hx)} with its minimum achieved at $\theta=0$, and strict positivity of $\tilde{\omega}$ is possible if $1<s<2$. 

We will first treat the case when $\tilde{\omega}$ is strictly positive. The following theorem describes how the ellipsoid-shaped energy minimizer expands as $\alpha$ gets large for a wider range of $s$.

\begin{theorem}\label{thm_abalpha1}
Assume $0< s < 3$, $\omega$ satisfies {\bf (hx)} (with the requirement $\min\omega=0$ relaxed to $\min\omega\ge 0$), and $\tilde{\omega}\ge c > 0$ ($\tilde{\omega}$ as given in Lemma \ref{lem_FT_1}). Then the parameters $a_\alpha,b_\alpha$ as given in Theorem \ref{thm_ell} for $\Omega_\alpha$, scale like $\alpha^{1/(2+s)}$ as $\alpha\rightarrow\infty$, i.e., $\lim_{\alpha\rightarrow\infty} a_\alpha \alpha^{-1/(2+s)}$ and $\lim_{\alpha\rightarrow\infty} b_\alpha \alpha^{-1/(2+s)}$ are positive numbers.
\end{theorem}

\begin{remark}
In the case $0<s\le 1$, the condition $\min\omega=0$ in {\bf (hx)} contradicts $\tilde{\omega}\ge c > 0$, by \eqref{lem_FT_4} (for $0<s<1$) and \eqref{lem_FT_4b} (for $s=1$). In other words, for $\omega$ satisfying {\bf (hx)}, this theorem is only non-vacuous for $1<s<3$. 

We also notice that for $\omega$ satisfying {\bf (hx)}, $\tilde{\omega}\ge c > 0$ is automatically satisfied for $2<s<3$ due to \eqref{lem_FT_2}, and conditionally satisfied for $1<s\le 2$ due to Theorem \ref{thm_s12exist} (for $1<s<2$) and \eqref{lem_FT_2b} (for $s=2$).
\end{remark}

\begin{proof}
For any $\alpha\in(0,\infty)$, the parameters $a_\alpha,b_\alpha\in (0,\infty)$ are uniquely determined by \eqref{abR3_2} for $\Omega_\alpha=1+\alpha\omega$ with $A=B=1$, i.e.,
\begin{align}\label{abR3_3}
    & \pi\tau_s R_1^{2+s}\int_0^\pi \sin^3\varphi (a_\alpha^2\sin^2\varphi+b_\alpha^2\cos^2\varphi)^{-(2+s)/2}(c_s+\alpha\tilde{\omega}(\varphi))\rd{\varphi} = 1 \\
    & 2\pi\tau_s R_1^{2+s}\int_0^\pi \cos^2\varphi\sin\varphi (a_\alpha^2\sin^2\varphi+b_\alpha^2\cos^2\varphi)^{-(2+s)/2}(c_s+\alpha\tilde{\omega}(\varphi))\rd{\varphi} = 1\,. \nonumber
\end{align}
Similar to the proof of Lemma \ref{lem_ab2}, for $t\in (0,\infty)$ we define the function
\begin{equation}\label{ftalpha}
    f(t,\alpha)=\frac{\int_0^\pi \sin^3\varphi (\sin^2\varphi+t^2\cos^2\varphi)^{-(2+s)/2}(c_s\alpha^{-1}+\tilde{\omega}(\varphi))\rd{\varphi}}{2\int_0^\pi \cos^2\varphi\sin\varphi (\sin^2\varphi+t^2\cos^2\varphi)^{-(2+s)/2}(c_s\alpha^{-1}+\tilde{\omega}(\varphi))\rd{\varphi}},
\end{equation}
which satisfies the relation $f(b_\alpha/a_\alpha,\alpha)=1$ for any $\alpha\in(0,\infty)$. We define $f(t,\infty)$ for $t\in (0,\infty)$ by 
\begin{equation}\label{ftinfty}
    f(t,\infty)=\frac{\int_0^\pi \sin^3\varphi (\sin^2\varphi+t^2\cos^2\varphi)^{-(2+s)/2}\tilde{\omega}(\varphi)\rd{\varphi}}{2\int_0^\pi \cos^2\varphi\sin\varphi (\sin^2\varphi+t^2\cos^2\varphi)^{-(2+s)/2}\tilde{\omega}(\varphi)\rd{\varphi}},
\end{equation}
which is well-defined since $\tilde{\omega}\ge c > 0$. For any $\alpha\in (0,\infty]$, $f(t,\alpha)=1$ has a unique solution $t_\alpha\in (0,\infty)$ (which is $b_\alpha/a_\alpha$ for $\alpha\in (0,\infty)$). The existence of solution to $f(t,\alpha)=1$ was shown in Lemma \ref{lem_ab2} which works for $\alpha\in (0,\infty]$, and the uniqueness was given in Remark \ref{rem_abuni} which works for $\alpha\in (0,\infty)$. The uniqueness can also be obtained by analyzing the derivative of $f(t,\alpha)$, and this approach works for $\alpha\in (0,\infty]$. In fact, after tedious but simple computations one gets
\begin{equation}
    \partial_t f(t,\alpha) = -(2+s)t \frac{ \Big(\int\sin^2\varphi\cos^2\varphi \Big)^2 - \int\sin^4\varphi\cdot \int\cos^4\varphi}{2[\int \cos^2\varphi(\sin^2\varphi+t^2\cos^2\varphi)]^2}\,,
\end{equation}
where the integrals are with respect to the positive weight 
$$
(\sin^2\varphi+t^2\cos^2\varphi)^{-(2+s)/2-1}(c_s\alpha^{-1}+\tilde{\omega}(\varphi))\sin\varphi\rd{\varphi}.
$$
Notice that the numerator is negative by the Cauchy-Schwarz inequality. Therefore, we see $\partial_t f(t,\alpha)>0$ for any $t\in (0,\infty)$ and $\alpha\in (0,\infty]$.

Since the denominator in \eqref{ftalpha} is positive and away from zero, it is clear that
\begin{equation}
    \lim_{\alpha\rightarrow\infty}f(t,\alpha)=f(t,\infty)
\end{equation}
and the limit is uniform on compact sets for $t\in (0,\infty)$. Therefore, using $\partial_t f(t,\infty)>0$, we may apply the implicit function theorem to conclude that 
\begin{equation}
    \lim_{\alpha\rightarrow\infty}t_\alpha = t_\infty
\end{equation}
i.e., $b_\alpha/a_\alpha\rightarrow t_\infty$ as stated in the theorem.

To see the limiting behavior of $a_\alpha$ as $\alpha\rightarrow\infty$ (which also implies the same for $b_\alpha$), we recall the $A=1$ equation in \eqref{abR3_3} (writing $b_\alpha=a_\alpha t_\alpha$)
\begin{equation}\label{eqtalpha}\begin{split}
    & \pi\tau_s R_1^{2+s}\int_0^\pi \sin^3\varphi (\sin^2\varphi+t_\alpha^2\cos^2\varphi)^{-(2+s)/2}(c_s\alpha^{-1}+\tilde{\omega}(\varphi))\rd{\varphi} = \alpha^{-1}a_\alpha^{2+s} \\
\end{split}\end{equation}
Sending $\alpha\rightarrow\infty$, the LHS converges to a positive constant. Therefore so does the RHS.

\end{proof}

\begin{remark}
This result shows that $a_\alpha,b_\alpha$ scale (as $\alpha\rightarrow\infty$) in the same way as if $\omega=1$, i.e., an isotropic potential whose repulsive part has strength $\alpha$. One can easily show that this agreement is also true for the minimal total energy, i.e., scales like $\alpha^{2/(2+s)}$ both the stated $\omega$ and $\omega=1$. In particular, if one considers $1<s<2$ and $\omega$ as in item 1 of Theorem \ref{thm_s12exist}, the smallness of $\omega$ near $\theta=0$ does not affect the minimal energy in terms of its $\alpha$-scaling.
\end{remark}

\begin{remark}
A similar result can also be proved for 2D anisotropic potentials studied in \cite{CS22}.
\end{remark}

In the case $1\le s<2$ with $\tilde{\omega}$ nonnegative but not having a positive lower bound, it is possible that $b_\alpha\rightarrow 0$ as $\alpha\rightarrow\infty$, as shown in the following theorem.

\begin{theorem}\label{thm_abalpha2}
Assume $1\le s < 2$. Then there exists $\omega$ satisfying {\bf (hx)}, achieving minimum only at $\omega(0)=0$, with $\tilde{\omega}\ge 0$, such that the parameters $a_\alpha,b_\alpha$ as given in Theorem \ref{thm_ell} for $\Omega_\alpha$, satisfying $b_\alpha\rightarrow 0$ as $\alpha\rightarrow\infty$.
\end{theorem}

This asymptotic behavior of minimizers is counter-intuitive, because $\omega$ achieves minimum only at $\omega(0)=0$ (i.e., the $x_3$-direction), but the minimizers expand in its perpendicular directions, the $x_1x_2$-plane.

\begin{proof}
We follow the notations in the proof of Theorem \ref{thm_abalpha1}. We take $\omega$ as in Lemma \ref{lem_s12exist2} with $\varphi_0$ small. Then $\omega$ satisfies {\bf (hx)}, achieving minimum only at $\omega(0)=0$, and $\tilde{\omega}$ is nonnegative and concentrated near $\varphi=\pi/2$. For such $\tilde{\omega}$, we have $f(t,\infty)$, as in \eqref{ftinfty}, well-defined and continuous for $t\in [0,\infty)$ as discussed in the previous proof. Moreover, one can ensure that $f(0,\infty)\ge 2$ if $\varphi_0$ is sufficiently small since the support of $\tilde\omega$ lies inside $\varphi\in [\pi/2-\varphi_0,\pi/2+\varphi_0]$. As a result, $f(t,\infty)\ge 2$ for any $t\in[0,\infty)$ since we showed that $f(t,\infty)$ is increasing in $t$.

Let $t=t_\alpha>0$ be the unique solution to $f(t,\alpha)=1$ (defined in \eqref{ftalpha}). Since Notice that 
\begin{equation}
    1=f(t_\alpha,\alpha) = \frac{f_{\textnormal{num}}(t_\alpha,\infty) + \alpha^{-1}\bar f_{\textnormal{num}}(t_\alpha)}{f_{\textnormal{denom}}(t_\alpha,\infty) + \alpha^{-1}\bar f_{\textnormal{denom}} (t_\alpha)}
\end{equation}
with
$$
\bar f_{\textnormal{num}}(t)= c_s\int_0^\pi \sin^3\varphi (\sin^2\varphi+t^2\cos^2\varphi)^{-(2+s)/2}\rd{\varphi}
$$
and
$$
\bar f_{\textnormal{denom}} (t)= 2c_s\int_0^\pi \cos^2\varphi\sin\varphi (\sin^2\varphi+t^2\cos^2\varphi)^{-(2+s)/2}\rd{\varphi}\,,
$$
and with $f_{\textnormal{num}}(t,\infty)$ and $f_{\textnormal{denom}}(t,\infty)$ denoting the numerator/denominator as in \eqref{ftinfty}. 
Note that
$$
\lim_{t\rightarrow\infty}\frac{\bar f_{\textnormal{num}}(t)}{\bar f_{\textnormal{denom}} (t)}=\infty.
$$
As a consequence, there exists $T$ independent of $\alpha$ such that this fraction is at least 2 for any $t\ge T$. Moreover, since $f(t,\infty) = \frac{f_{\textnormal{num}}(t,\infty)}{f_{\textnormal{denom}}(t,\infty)}\ge 2$ for any $t\in [0,\infty)$, then we conclude that $t_\alpha\in (0,T)$ for any $\alpha\ge 0$. Indeed, in case $t_\alpha\geq T$, we have
$$
f(t_\alpha,\alpha) 
    \ge \frac{2f_{\textnormal{denom}}(t_\alpha,\infty) + 2\alpha^{-1}\bar f_{\textnormal{denom}} (t_\alpha)}{f_{\textnormal{denom}}(t_\alpha,\infty) + \alpha^{-1}\bar f_{\textnormal{denom}} (t_\alpha)} = 2
$$
since both $\bar f_{\textnormal{num}}(t_\alpha)\geq 2 \bar f_{\textnormal{denom}} (t_\alpha)$ and $f_{\textnormal{num}}(t_\alpha,\infty)\geq 2 f_{\textnormal{denom}}(t_\alpha,\infty)$, contradicting the definition of $t_\alpha$.

Both $f_{\textnormal{num}}(t,\infty)$ and $f_{\textnormal{denom}}(t,\infty)$ are positive and bounded from below for any $t\in [0,T]$. Therefore, in order to have $f(t_\alpha,\alpha)=1$ with $t_\alpha\in (0,T)$, one necessarily has
\begin{align*}
   \alpha^{-1} \bar f_{\textnormal{denom}} (t_\alpha) &= \alpha^{-1}c_s\int_0^\pi \cos^2\varphi\sin\varphi (\sin^2\varphi+t_\alpha^2\cos^2\varphi)^{-(2+s)/2}\rd{\varphi} \\
   &\ge f_{\textnormal{num}}(t_\alpha,\infty)- f_{\textnormal{denom}}(t_\alpha,\infty)\geq c\,,
\end{align*}
which implies
\begin{equation}\label{talpha_scale}
    t_\alpha \le C\alpha^{-1/s}
\end{equation}
by analyzing the asymptotic behavior of the above integral for small $t_\alpha$. In fact, this integral has a singularity near $\varphi=0$ for small $t_\alpha$. By separating the integral into two parts with $|\sin\varphi|\le t_\alpha$ and $|\sin\varphi|> t_\alpha$, it is clear that it behaves like $t_\alpha^{-s}$, and the desired scaling in \eqref{talpha_scale} is obtained.

\eqref{eqtalpha} still gives $a_\alpha \sim \alpha^{1/(2+s)}$ since the LHS integral converges to the positive constant $\int_0^\pi (\sin\varphi)^{1-s}\tilde{\omega}(\varphi)\rd{\varphi}$ as $\alpha\rightarrow\infty$ (noticing $1\le s<2$). Therefore we see that $b_\alpha \lesssim \alpha^{-1/s+1/(2+s)}$ which converges to zero as $\alpha\rightarrow\infty$.

\end{proof}

\subsection{The general case: expansion of minimizers for large $\alpha$}

In the general case when $1\le s<2$, $\omega$ achieving its minimum at $\theta=0$ and $\tilde{\omega}$ is possibly sign-changing, we will apply delicate comparison arguments to show that any energy minimizer cannot be constrained in a fixed infinite cylinder as $\alpha$ increases. In other words, the minimizer has to expand in at least two dimensions. 

We also notice that the expansion in all three dimensions is not true in general, due to the example in Theorem \ref{thm_abalpha2}.

\begin{theorem}\label{thm_cylinder}
Assume $1\le s < 2$ and $\omega$ satisfies {\bf (hx)} with $\omega(\pi/2)>0$.  If $\rho_{(\alpha)}$ is an energy minimizer  for $W_\alpha$ with zero center of mass for each $\alpha>0$,  then for any $\xi\in S^2$,
\begin{equation}\label{thm_cylinder_1}
    \lim_{\alpha\rightarrow\infty}\sup_{\bx\in\supp\rho_{(\alpha)}} \sqrt{|\bx|^2-|\bx\cdot\xi|^2} = \infty\,.
\end{equation}
\end{theorem}
Here $\sqrt{|\bx|^2-|\bx\cdot\xi|^2}$ is the distance between $\bx$ and the line containing the vector $\xi$. Therefore the result shows that $\supp\rho_{(\alpha)}$ cannot be contained in a fixed infinite cylinder for all $\alpha$. Let us define $\mathcal{C}_{R,\xi}$ to be the cylinder of radius $R>0$ and direction $\xi\in S^2$ defined by the inequality $\sqrt{|\bx|^2-|\bx\cdot\xi|^2}\le R$.

The assumption $\omega(\pi/2)>0$ is sharp, because Lemma \ref{lem_exist} shows that if $\omega(\pi/2)=0$ then any energy minimizer  for $W_\alpha$ with zero center of mass has uniformly bounded support.

To prove this theorem, we will argue by contradiction and assume that every $\rho_{(\alpha)}$ is supported in a fixed cylinder of radius $R$. Then we will use a comparison argument with isotropic energies, based on the following lemma.
\begin{lemma}\label{lem_cEalpha}
Assume $1\le s < 2$. For any $\alpha> 0$, 
let $\cE_\alpha$ denote the total energy for $\cW_\alpha(\bx)=\alpha|\bx|^{-s}+|\bx|^2$. Fix $R>0$ and $\xi\in S^2$. Denote $\mathcal{P}$ as the set of probability measures on $\mathbb{R}^3$, and
\begin{equation}
    \mathcal{P}_{R,\xi}=\Big\{\rho\in\mathcal{P}: \rho\textnormal{ compactly supported, } \int_{\mathbb{R}^3}\bx\rho\rd{\bx}=0,\, \supp\rho\subset \mathcal{C}_{R,\xi}\Big\}.
\end{equation} 
Then
\begin{equation}
    \inf_{\rho\in\mathcal{P}_{R,\xi}} \cE_\alpha[\rho] \ge\left\{\begin{split} & c \alpha^{2/3},\quad 1<s<2 \\
    & c(\alpha\ln \alpha)^{2/3},\quad s=1
    \end{split}\right.
\end{equation}
for any $\alpha\ge 1$, where $c$ only depends on $s$ and $R$.
\end{lemma}

Notice that the minimal value of $\cE_\alpha[\rho]$ for $\rho\in \mathcal{P}$ scales like $\alpha^{2/(2+s)}$ where the exponent $2/(2+s) < 2/3$ for $1<s<2$.  Therefore the above result shows that restricting to an infinite cylinder of radius $R$ makes the minimal energy larger in terms of its $\alpha$-scaling (which degenerates to a logarithmic factor for $s=1$).

\begin{proof}

We first treat the case $1<s<2$. Let $\rho\in \mathcal{P}_{R,\xi}$. Define
\begin{equation}
    a_n = \int_{|\bx\cdot\xi|\in [n-1/2,n+1/2)}\rho\rd{\bx},\quad n\in\mathbb{Z}
\end{equation}
as the mass of $\rho$ inside each piece of the cylinder of height 1. Then the repulsive part of the energy satisfies
\begin{equation}
    \int_{\mathbb{R}^3} (|\cdot|^{-s}*\rho)\rho\rd{\bx} \ge \sum_n \iint_{|\bx\cdot\xi|,|\by\cdot\xi|\in [n-1/2,n+1/2)} |\bx-\by|^{-s}\rho(\by)\rd{\by}\rho(\bx)\rd{\bx} \ge c\sum_n a_n^2
\end{equation}
since one always has $|\bx-\by|^{-s}\ge c$ in the last integrand, $c$ depending on $R$. The quadratic part of the energy is estimated by
\begin{equation}
    \int_{\mathbb{R}^3} (|\cdot|^2*\rho)\rho\rd{\bx} = 2\int_{\mathbb{R}^3} |\bx|^2\rho(\bx)\rd{\bx} \ge c\sum_n n^2a_n
\end{equation}
using the mean-zero assumption. Therefore we see that
\begin{equation}
    \cE_\alpha[\rho] \ge c\left(\alpha \sum_n a_n^2 + \sum_n n^2a_n\right) .
\end{equation}
If we denote
\begin{equation}\label{Falpha}
    F_\alpha[u] = \alpha\int_{\mathbb{R}} u^2\rd{x} + \int_{\mathbb{R}} x^2u(x)\rd{x}
\end{equation}
for nonnegative $u\in L^2(\mathbb{R})$ with compact support and $\int_{\mathbb{R}}u\rd{x}=1$, then it is clear that
\begin{equation}
    \alpha \sum_n a_n^2 + \sum_n n^2a_n \ge c F_\alpha\Big[\sum_n a_n\chi_{[n-1/2,n+1/2)}\Big]-1 \ge c \inf_u F_\alpha[u]-1,
\end{equation}
where the extra `$-1$' takes account of the $n=0$ term in $\sum_n n^2a_n$. It is straightforward to show that $\inf_u F_1[u]>0$, and $\inf_u F_\alpha[u]=\alpha^{2/3}\inf_u F_1[u]$ by rescaling. Therefore we conclude that
\begin{equation}
    \cE_\alpha[\rho] \ge c(\alpha^{2/3}-1) ,
\end{equation}
which implies the desired result for any sufficiently large $\alpha$. The case of smaller $\alpha\ge 1$ is clearly true up to switching to a smaller constant $c$.

Finally we treat the case $s=1$ which is more delicate. We will imitate the previous proof but consider cylinder pieces of various heights. Define
\begin{equation}
    a_{k,n} = \int_{|\bx\cdot\xi|\in [2^k(n-1/2),2^k(n+1/2))}\rho\rd{\bx},\quad k\in\mathbb{Z}_{\ge 0},\,n\in\mathbb{Z}
\end{equation}
as the mass of $\rho$ inside each piece of the cylinder of height $2^k$. Notice that
\begin{equation}
    |\bx|^{-1} \ge c\sum_{k=0}^\infty 2^{-k}\chi_{|\bx|\le (1+2R)2^k}
\end{equation}
with $c$ depending on $R$. It is clear that $|\bx_1-\bx_2|\le (1+2R)2^k$ whenever $\bx_1,\bx_2$ are in the same piece of the cylinder of height $2^k$. Therefore
\begin{align*}
    \int_{\mathbb{R}^3} (|\cdot|^{-1}*\rho)\rho\rd{\bx}&\ge c\sum_{k=0}^\infty 2^{-k}\int_{\mathbb{R}^3} (\chi_{|\cdot|\le (1+2R)2^k}*\rho)\rho\rd{\bx} \\
    & \ge c\sum_{k=0}^\infty 2^{-k}
\sum_n \iint_{|\bx\cdot\xi|,|\by\cdot\xi|\in [2^k(n-1/2),2^k(n+1/2))}\!\!\!\!\!\!\!\!\!\!\!\!\!\!\!\!\!\!\!\!\! \rho(\by)\rd{\by}\rho(\bx)\rd{\bx}    \\
    &\ge c\sum_{k=0}^\infty 2^{-k}\sum_n a_{k,n}^2 \,.
\end{align*}
For any $k\in\mathbb{Z}_{\ge 0}$, the quadratic part of the energy is estimated by
\begin{equation}
    \int_{\mathbb{R}^3} (|\cdot|^2*\rho)\rho\rd{\bx} = 2\int_{\mathbb{R}^3} |\bx|^2\rho(\bx)\rd{\bx} \ge c\sum_n 2^{2k}n^2a_{k,n}
\end{equation}
using the mean-zero assumption (with $c$ independent of $k$). Therefore we see that
\begin{equation}
    \cE_\alpha[\rho] \ge \frac{c}{K}\sum_{k=0}^{K-1}\Big(\alpha K 2^{-k}\sum_n a_{k,n}^2 + \sum_n 2^{2k}n^2a_{k,n}\Big) 
\end{equation}
for any $K\in\mathbb{Z}_{\ge 1}$. Later we will specify the choice of $K$.

Using the notation $F_\alpha[u]$ in \eqref{Falpha}, for every $k\in\mathbb{Z}_{\ge 0}$, we have
\begin{equation}
    \alpha K 2^{-k}\sum_n a_{k,n}^2 + \sum_n 2^{2k}n^2a_{k,n} \ge c\Big(F_{\alpha K}\Big[\sum_n a_{k,n}2^{-k}\chi_{[2^k(n-1/2),2^k(n+1/2))}\Big]-2^{2k}\Big)
\end{equation}
where the extra `$-2^{2k}$' takes account of the $n=0$ term in $\sum_n 2^{2k}n^2a_{k,n}$. Then, using $\inf_u F_\alpha[u]=c\alpha^{2/3}$ as before, we obtain
\begin{equation}
    \alpha K 2^{-k} \sum_n a_{k,n}^2 + \sum_n 2^{2k}n^2a_{k,n} \ge c(c(\alpha K)^{2/3}-2^{2k})\,.
\end{equation}
Therefore, we conclude
\begin{equation}
    \cE_\alpha[\rho] \ge \frac{c}{K}\sum_{k=0}^{K-1}(c(\alpha K)^{2/3}-2^{2k})  \ge c\alpha^{2/3}K^{2/3}-C 2^{2K} \,.
\end{equation}
By taking $K \sim \ln \alpha$ with properly chosen constant multiple, one can absorb the last negative term and obtain the conclusion.

\end{proof}

\begin{proof}[Proof of Theorem \ref{thm_cylinder}]
Since $\omega$ is smooth with $\omega(\pi/2)>0$, there exists $\theta_0\in (0,\pi/2)$ such that $\omega(\theta)\ge c > 0$ for any $\theta\in [\pi/2-\theta_0,\pi/2]$. Using the construction $\omega_1$ in Theorem \ref{thm_s12exist}, we may define
\begin{equation}
    \omega_4(\theta)=\omega_1(\theta)-\epsilon_4
\end{equation}
for sufficiently small $\epsilon_4>0$, such that $\omega_4(\theta)<0$ for any $\theta\in [0,\pi/2-\theta_0]$ while $\tilde{\omega}_4>0$. Here, we use that Theorem \ref{thm_s12exist} ensures $\tilde{\omega}_1>0$. Define
\begin{equation}
    \omega_*(\theta) = \omega(\theta)-\epsilon \omega_4(\theta)
\end{equation}
where $\epsilon>0$ is sufficiently small so that $\omega_*(\theta)> 0$ for any $\theta\in [\pi/2-\theta_0,\pi/2]$. We also have $\omega_*(\theta)>0$ for any $\theta\in [0,\pi/2-\theta_0]$ by the choice of $\omega_4$. Therefore, $\omega_*\ge c > 0$.

Recall that $\rho_{(\alpha)}$ is an energy minimizer for $W_\alpha(\bx)=|\bx|^{-s}(1+\alpha \omega(\bbx))+|\bx|^2$. Therefore, since $1+\alpha \omega \le C\alpha$ for any $\alpha\ge 1$, we have the upper bound for the minimal energy
\begin{equation}\label{Ealpha1}
    E_\alpha[\rho_{(\alpha)}] = \min_\rho E_\alpha[\rho] \le C\min_\rho\cE_\alpha[\rho] = C \alpha^{2/(2+s)}
\end{equation}
using the isotropic energy functional $\cE_\alpha$ defined in Lemma \ref{lem_cEalpha}. On the other hand, we define the energy for $W_{\alpha,*}(\bx)=|\bx|^{-s}(1+\alpha \omega_*(\bbx))+|\bx|^2$ as
\begin{equation}\begin{split}
    E_{\alpha,*}[\rho] = & \frac{1}{2}\int_{\mathbb{R}^3}\int_{\mathbb{R}^3} W_{\alpha,*}(\bx-\by)\rho(\by)\rd{\by}\rho(\bx)\rd{\bx} \\
    = & E_\alpha[\rho] -  \frac{\epsilon\alpha}{2}\int_{\mathbb{R}^3}\int_{\mathbb{R}^3} |\bx-\by|^{-s}\omega_4(\overline{\bx-\by})\rho(\by)\rd{\by}\rho(\bx)\rd{\bx}
\end{split}\end{equation}
where the last integral is well-defined as long as $E_\alpha[\rho]<\infty$, due to dominated convergence. Then
\begin{equation}\label{Ealpha2}
    E_{\alpha,*}[\rho] \le E_\alpha[\rho]
\end{equation}
for any compactly supported $\rho\in\mathcal{P}$ with $E_\alpha[\rho]<\infty$,  since $\cF[|\bx|^{-s}\omega_4(\bbx)] = |\xi|^{-3+s}\tilde{\omega}_4(\bxi) > 0$ by the construction of $\omega_4$. Since 
$\omega_*\ge c > 0$, we have
\begin{equation}\label{Ealpha3}
    E_{\alpha,*}[\rho_{(\alpha)}] \ge c \cE_\alpha[\rho_{(\alpha)}]
\end{equation}
for any $\alpha\ge 1$.

To prove \eqref{thm_cylinder_1}, assume on the contrary that $\supp\rho_{(\alpha_n)}\subset \mathcal{C}_{R,\xi} $ for some $R>0$ and a sequence $\{\alpha_n\}$ with $\lim_{n\rightarrow\infty}\alpha_n=\infty$. In other words, $\rho_{(\alpha_n)} \in \mathcal{P}_{R,\xi}$ for every $n$. Then Lemma \ref{lem_cEalpha} gives
\begin{equation}
    \cE_{\alpha_n}[\rho_{(\alpha_n)}] \ge \left\{\begin{split} & c \alpha_n^{2/3},\quad 1<s<2 \\
    & c(\alpha_n\ln \alpha_n)^{2/3},\quad s=1
    \end{split}\right.\,.
\end{equation}
However, a combination of \eqref{Ealpha1}, \eqref{Ealpha2} and \eqref{Ealpha3} gives
\begin{equation}
    \cE_{\alpha_n}[\rho_{(\alpha_n)}] \le C\alpha_n^{2/(2+s)}\,.
\end{equation}
Therefore, for any $1\le s < 2$, we obtain a contradiction for sufficiently large $n$.

\end{proof}

\appendix

\section{List of notations and integral formulas}
\label{app:constants}

According to \cite{CV1,CDM16,CH,carrilloshu21} (with suitable rescaling), the unique energy minimizer for the interaction potential $|\bx|^{-s}+|\bx|^2$ in $d$-dimension with $0<s<d$ (in the class of probability measures, up to translation) is given explicitly by
\begin{equation}\label{rhod}
    \rho_d(\bx) = C_d(R_d^2-|\bx|^2)_+^{(s+2-d)/2},\quad d=1,2,3,\quad 0<s<d. 
\end{equation}
Here the explicit formulas for $R_d$ and $C_d$ are given by
\begin{equation}\label{R1C1}\begin{split}
    & R_1 = \Big(\frac{2\cos\frac{s\pi}{2}}{s(s+1)\pi} \beta\Big(\frac{1}{2},\frac{3+s}{2}\Big)\Big)^{\frac{1}{-s-2}},\quad C_1=\frac{2\cos\frac{s\pi}{2}}{s(s+1)\pi} \\
    & R_2 = \Big(\frac{8\sin\frac{s\pi}{2}}{s^2(2+s)\pi}\Big)^{\frac{1}{-s-2}},\quad C_2 = \frac{4\sin\frac{s\pi}{2}}{s^2\pi^2} \\
    & R_3 = \Big(\frac{6\cos\frac{s\pi}{2}}{s(1-s)\pi} \beta\Big(\frac{3}{2},\frac{1+s}{2}\Big)\Big)^{\frac{1}{-s-2}},\quad C_3=\frac{3\cos\frac{s\pi}{2}}{s(1-s)\pi} \\
\end{split}\end{equation}
where $\beta$ denotes the Beta function. Here $R_3,C_3$ are well-defined and positive for the whole range $0<s<3$ because the point $s=1$ is a removable singularity.

We denote the following probability measures on $\mathbb{R}^3$
\begin{equation}\label{rho1d2d}
    \rho_{\textnormal{1D}}(\bx) = \delta(x_1)\delta(x_2)\rho_1(x_3),\quad \rho_{\textnormal{2D}}(\bx) = \delta(x_3)\rho_2(x_1,x_2)
\end{equation}
as candidates of lower-dimensional energy minimizers.

The Fourier transform of power functions on $\mathbb{R}^3$ is given by
\begin{equation}\label{calc2}
    \cF[|\bx|^{-s}] = c_s |\xi|^{-3+s},\quad 0<s<3,\quad c_s = \pi^{s-\frac{3}{2}}\frac{\Gamma((3-s)/2)}{\Gamma(s/2)}.
\end{equation}

In the sense of improper integral, we can obtain
\begin{equation}\label{calc3}
    \int_0^\infty r^{s}\cos r   \rd{r} = -\Gamma(1+s)\sin\frac{s\pi}{2},\quad -1<s<0
\end{equation}
and
\begin{equation}\label{calc3s}
    \int_0^\infty r^{s}\sin r   \rd{r} = \Gamma(1+s)\cos\frac{s\pi}{2},\quad -1<s<0.
\end{equation}
These two formulas can be proved by contour integrals whose details are omitted.

Finally, we denote 
\begin{equation}\label{calc4}
    \tau_s = (2\pi)^{-s}\Gamma(s)\cos\frac{s\pi}{2},\quad 0<s<3 .
\end{equation}
Notice that $\tau_1=0$, and $\tau_s$ is negative for $1<s<3$ and positive for $0<s<1$.

\section{Pointwise formula for the Fourier transform for $1<s< 2$}\label{app_FT}

\begin{lemma}
Under the same assumptions and notations as Lemma \ref{lem_FT}, we have
\begin{equation}\label{lem_FT_3}
    \tilde{\Omega}(0,0,1) = \tau_{3-s}\int_0^\pi\int_0^{2\pi}(\Omega(\bbx)-[\Omega]_{(0,0,1)})\rd{\mu}|\cos\theta|^{-3+s}\sin\theta\rd{\theta} + c_s [\Omega]_{(0,0,1)},\quad 1<s\le 2
\end{equation}
where the above integral is understood as an iterated integral.
\end{lemma}

The formula for $\tilde{\Omega}(\bxi)$ for $1<s\le 2$ and general $\bxi$ can be obtained by applying a rotation to \eqref{lem_FT_3}, but we do not give it explicitly because the notation would be cumbersome. Also notice that \eqref{lem_FT_2b} for $\bxi=(0,0,1)$ can be obtained as a special case of \eqref{lem_FT_3}.

\begin{proof}

The RHS of \eqref{lem_FT_3} is well-defined for any $s$ with $\Re(s)\in (1,3)$. To see this, we notice that $\theta\mapsto \int_0^{2\pi}(\Omega(\bbx)-[\Omega]_{(0,0,1)})\rd{\mu}$ is a smooth function of $\theta\in [0,\pi]$ which vanishes at $\theta=\pi/2$, by the definition of $[\Omega]_{(0,0,1)}$. Therefore $\int_0^{2\pi}(\Omega(\bbx)-[\Omega]_{(0,0,1)})\rd{\mu}|\cos\theta|^{-3+s}\sin\theta$ is integrable in $\theta$ for complex number $s$ with $\Re(s)\in (1,3)$. The RHS of \eqref{lem_FT_3} is holomorphic in $s$ because $\tau_{3-s}$ and $c_s$ are holomorphic, and one can take $s$-derivative of the integral as
\begin{equation}\begin{split}
    \partial_s\int_0^\pi\int_0^{2\pi} & (\Omega(\bbx)-[\Omega]_{(0,0,1)})\rd{\mu}|\cos\theta|^{-3+s}\sin\theta\rd{\theta} \\ = & \int_0^\pi\int_0^{2\pi}(\Omega(\bbx)-[\Omega]_{(0,0,1)})\rd{\mu}|\cos\theta|^{-3+s}\ln|\cos\theta|\sin\theta\rd{\theta}
\end{split}\end{equation}
since the RHS is integrable.

We also notice that the RHS of \eqref{lem_FT_3} agrees with $\tilde{\Omega}((0,0,1);s)$ for $s\in (2,3)$. In fact, by comparing with the formula \eqref{lem_FT_2}, it suffices to show that the two terms involving $[\Omega]_{(0,0,1)}$ cancel each other for $s\in (2,3)$. To see this, we use the relation 
\begin{equation}
    2\pi\tau_{3-s}\int_0^\pi|\cos\theta|^{-3+s}\sin\theta\rd{\theta} =2\pi(2\pi)^{-3+s}\Gamma(3-s)\cos\frac{(3-s)\pi}{2}\cdot\frac{2}{s-2} = c_s,
\end{equation} 
where the last inequality uses the formulas $\Gamma(z)\Gamma(1-z)=\frac{\pi}{\sin \pi z}$ and $\Gamma(z)\Gamma(z+1/2)=2^{1-2z}\sqrt{\pi}\Gamma(2z)$. Since the RHS of \eqref{lem_FT_3} and $\tilde{\Omega}((0,0,1);s)$ are both holomorphic in $s$ for $\Re(s)\in (1,3)$ and agree for $s\in (2,3)$, they have to agree for any $\Re(s)\in (1,3)$, which implies \eqref{lem_FT_3}.

\end{proof}

\section*{Acknowledgements}
JAC and RS were supported by the Advanced Grant Nonlocal-CPD (Nonlocal PDEs for Complex Particle Dynamics: Phase Transitions, Patterns and Synchronization) of the European Research Council Executive Agency (ERC) under the European Union's Horizon 2020 research and innovation programme (grant agreement No. 883363). JAC was also partially supported by the EPSRC grant number EP/T022132/1 and EP/V051121/1.

\bibliographystyle{abbrv}
\bibliography{biblio}

\end{document}